\documentclass[a4paper,12pt]{article} 
\usepackage[T1]{fontenc}
\usepackage{lmodern} 
\usepackage[pagewise]{lineno}
\usepackage{geometry}
\usepackage{amsmath}
\usepackage{amssymb}
\usepackage{indentfirst}
\usepackage{graphicx}
\usepackage{color}
\usepackage[colorlinks=true]{hyperref}
\usepackage{mathrsfs}
\usepackage{chngcntr}
\usepackage[english]{babel} 
\usepackage{appendix}
\newtheorem{Thm}{Theorem}[section]

\newtheorem{Lem}[Thm]{Lemma}

\newtheorem{Exam}{Example}[section]
\newtheorem{Rem}{Remark}[section]
\newtheorem{Def}{Definition}[section]

\newcommand{\Z}{\mathbb{Z}}
\newcommand{\R}{\mathbb{R}}

\numberwithin{equation}{section}

\newenvironment{proof}{\medskip\par\noindent{\bf Proof\/}.~}{\qquad
	\raisebox{-0.5mm}{\rule{1.5mm}{1mm}}\vspace{6pt}}

\title{Your Paper Title}
\author{}
\date{} 

\begin{document}
			
	\title{Fully sign-changing Nehari constraint vs sign-changing solutions of a competitive Schr\"{o}dinger system\thanks{This work is supported by NSFC 12261107, China.}}
	\author{Xuejiao Fu$^1$, Fukun Zhao$^{1,2}$\thanks{Corresponding authors. 
			E-mail addresses: fukunzhao@163.com (F. Zhao).}\\
		{\small 1. Department of Mathematics, Yunnan Normal University, Kunming 650500, P.R. China}\\
		{\small 2. Yunnan Research Center Of Modern Analysis And Partial Differential Equations,}\\ {\small Kunming 650500, PR China}}

	\date{ }
\maketitle

\begin{abstract}
	We study a competitive nonlinear Schr\"odinger system in $\mathbb{R}^N$ whose
	nonlinear potential is localized in small regions that shrink to isolated points. Within a variational framework based on a fully sign-changing Nehari constraint and
Krasnosel'skii genus, we construct, for all $\varepsilon>0$, a sequence of sign-changing solutions with increasing and unbounded energies, and 
after suitable translations they converge to a sequence of sign-changing
	solutions of the associated limiting system as $\varepsilon\to 0$ in $H^1$-norm. Moreover, these sign-changing solutions concentrate around the prescribed attraction points both in $H^1$-norm and $L^q$-norm for $q\in [1,\infty]$.
\end{abstract}

{\small \noindent{\bf  Keywords.}~Schr\"{o}dinger system; Sign-changing solution;  Shrinking region; Attraction; Concentration.}

\section{Introduction}\label{introduction}
\par 
In this paper we consider the following competitive Schr\"odinger system
\begin{equation}\label{eq1.1}
	\begin{cases}
		-\Delta u_i + u_i
		= \mu_i Q_\varepsilon(x-y_i) |u_i|^{2p-2}u_i
		+ \sum_{j\neq i} \lambda_{ij} |u_j|^{p} |u_i|^{p-2}u_i, \\[6pt]
		u_i \in H^1(\mathbb{R}^N),\ i=1,\dots,m.
	\end{cases}	
\end{equation}
where $N\ge1$, $m\ge2$. Nonlinear Schr\"odinger systems with competing interactions have been extensively studied in recent years, motivated both by physical models such as Bose--Einstein
condensates and nonlinear optics, and by the rich mathematical structure of coupled
elliptic equations.
Among the various phenomena that may occur in this context, a
particularly interesting situation arises when the nonlinear potentials are localized in
small regions that shrink to isolated points as a small parameter $\varepsilon>0$ tends to zero. In this case,
solutions may concentrate around finitely many prescribed points and exhibit a delicate
interaction between the different components.
\par
 We now state the structural assumptions that will be used throughout the paper. They are standard in this context and reflect the subcritical regime and the competitive nature of the system.
\begin{itemize}
  \item [$(A_1)$] $1<p<\frac{2^*}{2}$ if $N\ge3$, and $p>1$ if $N=1,2$,
where $2^*=\frac{2N}{N-2}$ if $N\ge3$.
  \item [$(A_2)$] $\mu_i>0$, $i=1,\dots,m$, $\lambda_{ij}=\lambda_{ji}<0$, for all $i\neq j$.
\item[$(A_3)$]
Let $Q_\varepsilon(x-y_i):=Q(\varepsilon x-y_i)$, where
$Q\in C(\mathbb{R}^N)$ is nonnegative, $\operatorname{supp}(Q)$ is bounded, and there exist exactly pairwise distinct points $y_1,\dots,y_m\in\mathbb{R}^N$ such that
\[
Q(-y_i)=\|Q\|_\infty=\max\limits_{x\in\mathbb{R}^N}Q(x)~\text{for all }i=1,\dots,m.
\]
\end{itemize}
\par
In this setting, the solutions $u_i$ may be interpreted as standing wave profiles of
different species which are attracted to regions where $Q_\varepsilon$ is
positive and repelled from their complement. The assumption $\lambda_{ij}<0$ means that
distinct components repel each other, which in turn favors spatial segregation, so the system \eqref{eq1.1} is competitive. 
\begin{Exam}\label{exam1.1}
There are functions satisfying $(A_3)$.
Assume $N\ge2$ and fix $m$ pairwise distinct points $y_1,\dots,y_m\in\mathbb R^N$.
Set
\[
d:=\min_{i\neq j}|y_i-y_j|>0,
~
r:=\frac{d}{4}.
\]
Let $\rho\in C_0^\infty(\mathbb R^N)$ be the standard compactly supported bump
\[
\rho(x):=
\begin{cases}
\exp\!\bigl(-\frac{1}{1-|x|^2}\bigr), & |x|<1,\\[2mm]
0, & |x|\ge 1,
\end{cases}
~
\text{and define}~
\phi(x):=\frac{\rho(x)}{\rho(0)} .
\]
Then $\phi\in C_0^\infty(\mathbb R^N)$, $\phi\ge0$, $\operatorname{supp}(\phi)\subset \overline{B_1(0)}$,
$\phi(0)=1$, and $\phi(x)<1$ for all $x\neq0$.
		
For each $i=1,\dots,m$ define
\[
\phi_i(x):=\phi\!\left(\frac{x+y_i}{r}\right),
~ x\in\mathbb R^N,
\]
and set
\[
Q(x):=\max_{1\le i\le m}\phi_i(x),~ x\in\mathbb R^N.
\]
Then $Q$ satisfies $(A_3)$.
\end{Exam}
\par
In the scalar case, a fundamental contribution is due to Ackermann and Szulkin~\cite{AckermannSzulkin2013},
who showed that when the positive region of the nonlinear coefficient collapses to isolated
points, every nontrivial solution concentrates at one of these cores, and ground states
select a single core without splitting their mass.  In this framework, the sign structure of the nonlinearity is the sole driver of localization, without any periodicity or
symmetry assumption on the linear part. Since then, their approach has been extended
in various directions, including problems on the whole space and coupled systems;
see, e.g.,~\cite{ClappHernandezSantamariaSaldana2025,ClappSaldanaSzulkin2025,ZhongZou2014}
and the references therein. We point out that these works focus primarily on positive solutions or
the least energy solutions, for which the Mountain Pass Theorem (\cite{Willem1996}) and the
maximum principles are available. Very recently, in \cite{ClappSaldanaSzulkin2025}, Clapp, Salda\~na and Szulkin studied the system \eqref{eq1.1} with a single or multiple shrinking domains, the existence of nonnegative least energy solution and concentration behavior were obtained as $\varepsilon\to0$. Furthermore, the
authors characterized the limit profile and described how the components either decouple
or remain coupled depending on the geometry of the attraction centers. Related results
for scalar equations, including concentration of semiclassical states and the description
of their limit profiles, were obtained in earlier works such as
\cite{AckermannSzulkin2013,FangWang2020}, and see 
\cite{TavaresTerracini2012} and the references therein for phase separation phenomena of competitive systems. We refer \cite{Clapp-Pistoia-Saldana-2026-JMPA} to the existence of concentrating positive solutions via a Lyapunov Schmidt reduction strategy.
\par
However, much less is known about sign-changing solutions for systems of the form
\eqref{eq1.1}. Even for the scalar equation, constructing nodal solutions requires a refined variational approach based on nodal Nehari sets and careful control of the positive and negative parts of the solutions. As far as we known, there are only two papers \cite{ClappHernandezSantamariaSaldana2025,Clapp-Pistoia-Saldana-2026-JMPA} concerned with this topic. In \cite{ClappHernandezSantamariaSaldana2025}, Clapp, Hern\'andez-Santamar\'ia and Salda\~na obtained the existence and concentration of  nodal solutions via the nodal Nehari manifold method,  and characterized the symmetries and the polynomial decay of the least-energy nodal limiting profiles. In \cite{Clapp-Pistoia-Saldana-2026-JMPA}, Clapp, Pistoia and Salda\~na established the existence of concentrating nodal solutions via a Lyapunov Schmidt reduction method.
But for systems, the situation is substantially more
delicate, since one has to keep track simultaneously of the sign structure of each
component and of the competitive couplings between different components. It seems that there is no work concerned with the existence and concentration of sign-changing solutions of the system \eqref{eq1.1}.
\par
First, we state the meaning of sign-changing solutions of \eqref{eq1.1}.
\begin{Def}\label{def1.1}
A solution $\mathbf{u}=(u_1,u_2,\cdots,u_m)$ of \eqref{eq1.1} is called \emph{sign-changing} if for each $i=1,\dots,m$, both the positive part $u_i^+=\max\{u_i,0\}$ and the negative part $u_i^-=\min\{u_i,0\}$ are nonzero in $H^1(\mathbb{R}^N)$, i.e., $\|u_i^\pm\|_{H^1}>0$.
\end{Def}
\par
The purpose of
this paper is to develop a variational framework to seek for infinitely many sign-changing solutions of \eqref{eq1.1}, and to analyze their concentration behavior as $\varepsilon\to 0$. 
The main technical difficulties arise from the sign-changing nature of the solutions. 
First, the natural Nehari manifold associated with nonnegative solutions does
not capture directly sign changing solutions, so for elements in the Nehari manifold, one has to impose constraints separately on the
positive and negative parts of each component. This leads to a nonlinear and
nonconvex constraint set on which the direct minimization is not available.  
Second, the competitive couplings $\lambda_{ij}<0$ mix the components in a nontrivial way, and one must show that on the nodal constraint the energy still controls the
$H^{1}$-norm uniformly, so that Palais-Smale sequences are bounded.  
Third, compactness issues are more severe in the nodal setting, one has to
prevent vanishing of some sign components and to rule out loss of mass at
infinity while keeping track of the nodal structure. So it is interesting to seek for sign-changing solutions of \eqref{eq1.1} since the above difficulties prevent one to use techniques based on the maximum principles, the comparison arguments, or the monotonicity methods. 
These obstacles make it necessary to develop a variational strategy that is tailored to
the nodal structure of the problem and that remains effective in the presence of
competitive couplings. 
\par
We are now in a position to state the main results. Our first theorem
shows that, for small $\varepsilon>0$, system~\eqref{eq1.1} admits a sequence of sign-changing solutions with unbounded energies.
\begin{Thm}\label{thm1.1}
	Assume $N\ge1$, $m\ge2$, and \textbf{\textup{$(A_1)$--$(A_3)$}}. Then for $\varepsilon>0$, 
	\eqref{eq1.1} admits a sequence of sign-changing solutions
	\[
	\mathbf{u}^{(k)}_\varepsilon=(u^{(k)}_{\varepsilon,1},\dots,u^{(k)}_{\varepsilon,m})\in H, ~ k\in\mathbb{N}^*,
	\]
	satisfying
	\[
	0<J_\varepsilon\bigl(\mathbf{u}^{(1)}_\varepsilon\bigr)
	< \cdots < J_\varepsilon\bigl(\mathbf{u}^{(k)}_\varepsilon\bigr)
    < J_\varepsilon\bigl(\mathbf{u}^{(k+1)}_\varepsilon\bigr)
	< \cdots, ~\text{and}~ \lim\limits_{k\to\infty}
	J_\varepsilon\bigl(\mathbf{u}^{(k)}_\varepsilon\bigr)=+\infty.
	\]
\end{Thm}
\par
To treat sign-changing solutions of competitive Schr\"odinger systems with shrinking nonlinear potentials, we develop a variational framework tailored to the nodal setting of \eqref{eq1.1}, inspired by the ideas in
\cite{AckermannSzulkin2013,ClappSaldanaSzulkin2025} but adapted to the present
system. We work in the product space $H:=(H^1(\mathbb{R}^N))^m$ endowed with the natural norm, and we consider the energy functional associated with
\eqref{eq1.1} $J_\varepsilon:H\to\mathbb{R}$. Our first step is to introduce a fully sign-changing Nehari set
$\mathcal{N}^{\mathrm{sc}}_\varepsilon$ (see \S \ref{sec:framework}), by imposing Nehari type constraints
separately on the positive and negative parts of each component. On a suitable open
subset $\mathcal{A}$ of $H$, rather than on a
linear subspace, we define a nodal projection
\[
m_\varepsilon:\mathcal{A}\longrightarrow
\mathcal{N}^{\mathrm{sc}}_\varepsilon,
\]
and we use it to construct an even reduced functional
$\Psi_\varepsilon:=J_\varepsilon\circ m_\varepsilon$ on
$\mathcal{A}$. The symmetry of $\Psi_\varepsilon$ allows us to apply
Krasnosel'skii genus theory and to implement a symmetric minimax scheme, which provide infinitely many critical values and
	yield a sequence of nodal solutions with unbounded energies.
This feature is essential for applying genus theory in a genuinely nodal
setting and allows us to recover a symmetric minimax structure despite the strong
nonlinearity of the constraints. The proof of Theorem~\ref{thm1.1} will be carried out in Section~\ref{sec:existence}.

\medskip

The second result in this paper concerns the asymptotic behavior of the above sign-changing solutions as $\varepsilon\to 0^+$. After a suitable rescaling, the potentials $Q_\varepsilon$ converge to the constants $Q(-y_i)$, and one is naturally led to the limiting autonomous system
\begin{equation}\label{eq1.2}
	\begin{cases}
	-\Delta U_i + U_i
	= \mu_i Q(-y_i)\,|U_i|^{2p-2}U_i
	+ \sum_{j\neq i}\lambda_{ij}|U_j|^p|U_i|^{p-2}U_i,\\[6pt]
	 U_i\in H^1(\mathbb{R}^N),\ i=1,\dots,m.
\end{cases}
\end{equation}
The energy functional associated with
\eqref{eq1.2} is denoted by $J_0(\mathbf{u}): H\to\mathbb{R}$.
\begin{Thm}\label{thm1.2}
	Assume $N\ge1$, $m\ge2$, and \textup{(A$_1$)--(A$_3$)}.
	Let $\{\varepsilon_n\}\subset(0,+\infty)$ be a sequence with $\varepsilon_n\to 0$.
	For each fixed $k\in\mathbb{N}^*$, let $\{\mathbf{u}^{(k)}_{\varepsilon_n}\}$ be the sequence of
	sign-changing solutions given by Theorem~\ref{thm1.1}.
	Then there exist $m$ sequences $\{x^{(k)}_{\varepsilon_n,i}\}\subset\mathbb{R}^N$ and a sign-changing solution
	$\mathbf{U}^{(k)}=(U^{(k)}_1,\dots,U^{(k)}_m)\in H$ of the limiting system \eqref{eq1.2}
	such that, up to a subsequence,
	\begin{itemize}
		\item[(i)]
		$\varepsilon_n x_{\varepsilon_n,i}^{(k)}\to 0$ as $n\to\infty$ for $i=1,\dots,m$.
	In particular, for every $R>0$,
		\[
		Q\bigl(\varepsilon_n(\,\cdot+x_{\varepsilon_n,i}^{(k)})-y_i\bigr)\to Q(-y_i)
		~\text{uniformly in}~B_R(0),~i=1,\dots,m.
		\]
		
		\item[(ii)]
		$u^{(k)}_{\varepsilon_n,i}(\,\cdot + x^{(k)}_{\varepsilon_n,i})
		\longrightarrow U^{(k)}_i$ in $H^1(\mathbb{R}^N)$ as $n\to\infty$ for  $i=1,\dots,m$. 
		
	\item[(iii)]
	for every $q\in[1,\infty)$,
	\[
	\lim_{R\to+\infty}\,\limsup_{n\to\infty}
	\int_{\mathbb{R}^N\setminus B_{R/{\varepsilon_n}}(0)}
	|u^{(k)}_{\varepsilon_n, i}(x)|^q\,dx = 0,
	~ i=1,\dots,m,
	\]
	and for $q=\infty$,
	\[
	\lim_{R\to+\infty}\,\limsup_{n\to\infty}
	\|u^{(k)}_{\varepsilon_n, i}\|_{L^\infty(\mathbb{R}^N\setminus B_{R/{\varepsilon_n}}(0))} = 0,
	~ i=1,\dots,m.
	\]
		
		\item[(iv)]
		$\{\mathbf{U}^{(k)}\}$ are pairwise distinct and satisfy
		\[
		J_0\big(\mathbf{U}^{(1)}\big) < J_0\big(\mathbf{U}^{(2)}\big) < \cdots \to +\infty.
		\]
	\end{itemize}
\end{Thm}

\begin{Rem}\label{rem1.2}
	Theorem~\ref{thm1.2} extends the concentration results for nonnegative solutions obtained in
	\cite{ClappSaldanaSzulkin2025} to the sign-changing case.
	Assertions \textup{(i)}--\textup{(iii)} show that, after suitable translations, each component concentrates
	around the origin, while assertion \textup{(iv)} reveals a new phenomenon: the limiting autonomous system
	inherits infinitely many sign-changing solutions with unbounded energies, arising from the genus-based
	construction for the original problem. 
\end{Rem}

The proof of Theorem~\ref{thm1.2} is based on a concentration--compactness argument for suitable
translations of sign-changing solutions $\{\mathbf u_n\}$, which yields a nontrivial limiting profile solving
\eqref{eq1.2} and the strong convergence in (ii) via a Br\'ezis--Lieb type energy
decomposition.
The decay property in (iii) follows from localized estimates with cut--off functions, interpolation,
and Agmon-type exponential decay bounds; the case $q=\infty$ is obtained by a translation argument
combined with local elliptic regularity and the compact support of $Q$.
It is carried out in Section~\ref{sec:concentration}.

The paper is organized as follows.
In Section~\ref{sec:framework} we introduce the variational setting, the fully nodal Nehari set and the reduced functional, and state our assumptions on the nonlinear potentials.
In Section~\ref{sec:ps} we establish a compactness result for the reduced functional $\Psi_{\varepsilon}$, which will be crucial in the construction of sign-changing solutions.
Section~\ref{sec:existence} is devoted to a genus-based minimax scheme on $\mathcal N_{\varepsilon}^{\mathrm{sc}}$, and hence obtain infinitely many sign-changing critical points. In Section~\ref{sec:concentration} we prove the concentration result Theorem~\ref{thm1.2}, including the
description of the limiting profiles for the semiclassical states.
\medskip

\textbf{Notation.}
\begin{itemize}
	\item Let $w^+(x):=\max\{w(x),0\}$ and $w^-(x):=\min\{w(x),0\}$ be the positive and negative parts of $w$, respectively. 
	\item For $r>0$ and $x\in\mathbb{R}^N$, $B_r(x)$ is the open ball centered at $x$ with radius $r$.
    \item Boldface $\mathbf{u}$ denote a $\mathbb{R}^m$ vector-value function, $u_i$ denotes the $i$-th component function of $\mathbf{u}$. 
    \item "$\to$" denotes the strong convergence and "$\rightharpoonup$" denotes the weak convergence.
    \item For $k\in\Z^N$, we denote by $k*\mathbf u$ the integer translation of $\mathbf u$:
    \[
    (k*\mathbf u)(x):=\mathbf u(x-k),~ x\in\R^N .
    \]
    \item For a function space $E$, $E^m$ denotes the product space $\underbrace{E\times\cdots\times E}_{m}$.
  \end{itemize}

\section{Variational setting and the fully sign-changing Nehari constrain}\label{sec:framework}
\par
In this section, we introduce the variational framework and the fully sign-changing  Nehari set which will be used in the construction of sign-changing solutions. For $\varepsilon>0$, we define the family of shrinking potentials
\[
Q_\varepsilon(x-y_i):=Q(\varepsilon x-y_i), ~ i=1,\dots,m.
\]
Under \textbf{$(A_3)$}, the supports of $Q_\varepsilon(\cdot-y_i)$ concentrate and shrink to the
points $y_i$ as $\varepsilon\to 0^+$ for $i=1,2,\cdots,m$, creating $m$ distinct regions of attraction.
Let $H^1(\mathbb{R}^N)$ be the usual Sobolev space with the norm 
\[
\|v\|_{H^1}^2
:=\int_{\mathbb{R}^N}(|\nabla v|^2+|v|^2)dx,~v\in H^1(\mathbb{R}^N).
\]
 As we mentioned in Section~\ref{introduction}, we will work in the product Hilbert space
$H:=H^1(\mathbb{R}^N)^m$
equipped with the norm
\begin{equation*}
	\|\mathbf{u}\|^2
	:=
	\sum_{i=1}^m\int_{\mathbb{R}^N}\bigl(|\nabla u_i|^2+|u_i|^2\bigr)\,dx~\text{for}~\mathbf{u}=(u_1,\dots,u_m)\in H.
\end{equation*}
By the Sobolev embedding theorem, \textbf{$(A_1)$} implies that the embedding
\[
H^1(\mathbb{R}^N) \hookrightarrow L^{2p}(\mathbb{R}^N)
\]
is continuous and locally compact. The energy functional associated with \eqref{eq1.1} is 
\begin{equation}\label{eq2.1}
	\begin{aligned}
		J_\varepsilon(\mathbf{u})
		&= \frac12\|\mathbf{u}\|^2
		- \frac{1}{2p}\sum_{i=1}^m \mu_i
		\int_{\mathbb{R}^N} Q_\varepsilon(x-y_i)\,|u_i|^{2p}\,dx \\
		&\phantom{=}\;
		- \frac{1}{2p}\sum_{i\neq j}\lambda_{ij}
		\int_{\mathbb{R}^N}|u_i|^p |u_j|^p\,dx,~\mathbf{u}=(u_1,\dots,u_m)\in H.
	\end{aligned}
\end{equation}
Assumptions \textbf{$(A_1)$}--\textbf{$(A_3)$} and the embedding above imply that $J_\varepsilon$ is
well defined and of class $C^1(H,\mathbb{R})$. Moreover, $J_\varepsilon$ is even, i.e.\
$J_\varepsilon(-\mathbf{u})=J_\varepsilon(\mathbf{u})$ for all $\mathbf{u}\in H$.

A standard computation shows that 
\begin{equation}\label{eq2.2}
\begin{aligned}
	\langle J_\varepsilon^\prime(\mathbf{u}),\boldsymbol{\varphi}\rangle
	&= \sum_{i=1}^m \int_{\mathbb{R}^N}
	\bigl(\nabla u_i\cdot\nabla\varphi_i + u_i\varphi_i\bigr)\,dx
	- \sum_{i=1}^m \mu_i\int_{\mathbb{R}^N}
	Q_\varepsilon(x-y_i)|u_i|^{2p-2}u_i\varphi_i\,dx \\
	&\phantom{=}\;
	- \sum_{i\neq j}\lambda_{ij}\int_{\mathbb{R}^N}
	|u_j|^{p}|u_i|^{p-2}u_i\,\varphi_i\,dx,
\end{aligned}
\end{equation}
for all $\mathbf{u},\boldsymbol{\varphi}\in H$. Hence $\mathbf{u}\in H$ is a critical point of $J_\varepsilon$ if and only if it is a
weak solution of \eqref{eq1.1}. In particular, sign-changing solutions of
\eqref{eq1.1} correspond to sign-changing critical points of $J_\varepsilon$.
\par
Recall the standard Nehari manifold is defined by
\[
\mathcal{N}_\varepsilon
:=\{\mathbf{u}\in H\setminus\{\mathbf{0}\}:\langle J_\varepsilon^\prime(\mathbf{u}), \mathbf{u}\rangle=0\}.
\]
It is obvious that $\mathcal{N}_\varepsilon$ contains all nontrivial solutions of \eqref{eq1.1}, and it is usually used for seeking for positive or nonnegative solutions (see e.g. \cite{ClappSaldanaSzulkin2025} and \cite{Willem1996}). But this manifold does not work well for sign-changing solutions since the constrained minimizer on it not necessary to be sign-changing.  
To capture sign-changing solutions, we introduce a sign-changing analogue that imposes
orthogonality conditions separately on the positive and negative parts of each component. 
For $\mathbf{u}=(u_1,\ldots,u_m)\in H$, we denote by $u_i^\pm$ the positive and negative
parts of the $i$-th component $u_i$.
We start with an energy splitting formula for sign-decomposed functions.
\par
\begin{Rem}\label{rem2.1}
	Let $\mathbf u=(u_1,\dots,u_m)\in H$ and write $u_i=u_i^+ + u_i^-$ with $u_i^+u_i^-=0$ a.e.\ in $\mathbb R^N$,
	$i=1,\dots,m$. Set $\mathbf u^+=(u_1^+,\dots,u_m^+)$ and $\mathbf u^-=(u_1^-,\dots,u_m^-)$.
	Then
	\[
	J_\varepsilon(\mathbf u)
	=
	J_\varepsilon(\mathbf u^+)+J_\varepsilon(\mathbf u^-)
	-\frac1p\sum_{i\neq j}\lambda_{ij}\int_{\mathbb R^N}
	\Bigl(|u_i^+u_j^-|^p+|u_i^-u_j^+|^p\Bigr)\,dx .
	\]
\end{Rem}

To simplify the notations, we identify $u_i^\pm$ with the element
$(\underbrace{0,\ldots,0,u_i^\pm,0,\ldots,0}_{m})\in H$, and write
\[
\langle J^\prime_\varepsilon(\mathbf{u}),u_i^\pm\rangle
:= \langle J^\prime_\varepsilon(\mathbf{u}),(0,\ldots,0,u_i^\pm,0,\ldots,0)\rangle.
\]
\begin{Def}\label{def2.1}
	For $\varepsilon>0$ and $i\in\{1,\dots,m\}$, set
	\[
	\mathcal{N}^{\mathrm{sc}}_{\varepsilon,i}
	:=\Bigl\{\mathbf{u}\in H:\ u_i^+\neq 0,\ u_i^-\neq 0,\,	\langle J_\varepsilon^\prime(\mathbf{u}),u_i^+\rangle=\langle J_\varepsilon^\prime(\mathbf{u}),u_i^-\rangle=0\Bigr\}
	\]
	and define the fully sign-changing Nehari set by
	\[
	\mathcal{N}^{\mathrm{sc}}_\varepsilon
	:=\bigcap_{i=1}^m \mathcal{N}^{\mathrm{sc}}_{\varepsilon,i}.
	\]
\end{Def}
It is obvious that $\mathcal{N}^{\mathrm{sc}}_\varepsilon$ contains all sign-changing solutions $\mathbf{u}=(u_1,\cdots,u_m)$ with each component $u_i$ is sign-changing.  We will show below that
$\mathcal{N}_\varepsilon^{\mathrm{sc}}$ is nonempty (see Lemma~\ref{lem2.2}).
\par
By \eqref{eq2.2}, 
$\langle J_\varepsilon^\prime(\mathbf{u}),u_i^\pm \rangle=0$ can be written as
\begin{equation}\label{eq2.3}
	\|u_i^\pm\|_{H^1}^2
	= \mu_i\int_{\mathbb{R}^N}Q_\varepsilon(x-y_i)|u_i^\pm|^{2p}\,dx
	+ \sum_{j\neq i} \lambda_{ij}
	\int_{\mathbb{R}^N}|u_j|^p|u_i^\pm|^p\,dx,
\end{equation}
for each $i=1,\dots,m$.
Next we establish a coercivity estimate on $\mathcal{N}^{\mathrm{sc}}_\varepsilon$.

\begin{Lem}\label{lem2.1}
	There exists a constant $\alpha>0$, independent of $\varepsilon$, such that for every
	$\mathbf{u}\in\mathcal{N}^{\mathrm{sc}}_\varepsilon$ one has
	\[
	J_\varepsilon(\mathbf{u})\;\ge\;\alpha\,\|\mathbf{u}\|^2.
	\]
\end{Lem}

\begin{proof}
	For $\mathbf{u}\in\mathcal{N}^{\mathrm{sc}}_\varepsilon$, by
	\eqref{eq2.3}, for all $i$ and $\sigma\in\{+,-\}$,
	\[
	\|u_i^\sigma\|_{H^1}^2
	= \mu_i\int_{\mathbb{R}^N}Q_\varepsilon(x-y_i)|u_i^\sigma|^{2p}\,dx
	+ \sum_{j\neq i}\lambda_{ij}\int_{\mathbb{R}^N}|u_j|^p|u_i^\sigma|^p\,dx.
	\]
	Summing over $i$ and both signs and recalling that $u_i=u_i^+ + u_i^-$ yields
	\[
	\|\mathbf{u}\|^2
	= \sum_{i=1}^m \mu_i\int_{\mathbb{R}^N}Q_\varepsilon(x-y_i)|u_i|^{2p}\,dx
	+ \sum_{i\neq j}\lambda_{ij}\int_{\mathbb{R}^N}|u_i|^p|u_j|^p\,dx.
	\]
	Substituting this identity into \eqref{eq2.1}, we obtain
	\[
	\begin{aligned}
		J_\varepsilon(\mathbf{u})
		&= \Bigl(\frac12-\frac{1}{2p}\Bigr)\|\mathbf{u}\|^2
		- \frac{1}{2p}\sum_{i\neq j}\lambda_{ij}\int_{\mathbb{R}^N}|u_i|^p|u_j|^p\,dx \\
		&\ge \Bigl(\frac12-\frac{1}{2p}\Bigr)\|\mathbf{u}\|^2,
	\end{aligned}
	\]
	because $\lambda_{ij}<0$. The conclusion follows with
	$\alpha=\frac12-\frac{1}{2p}>0$.
\end{proof}
 
In particular, the functional $J_\varepsilon$  is coercive on  $\mathcal{N}_\varepsilon^{\mathrm{sc}}$ and controls the $H$-norm from above.
Since $\lambda_{ij}<0$, the right-hand side in \eqref{eq2.3} is strictly smaller
than in the decoupled case, which is favorable for coercivity along $\mathcal{N}^{\mathrm{sc}}_\varepsilon$.
\par
Fix $\mathbf{u}\in H$ such that $u_i^\pm\neq 0$ for all $i=1,\dots,m$. For each fixed $i$ we consider
the two parameter fiber map
\[
\Phi_i(t^+,t^-)
:= J_\varepsilon\bigl(u_1,\dots,t^+u_i^+ + t^- u_i^-,\dots,u_m\bigr),~(t^+,t^-)\in(0,\infty)^2.
\]
Since $u_i^+u_i^-=0$ a.e. in $\mathbb{R}^N$, the decomposition holds
\begin{equation}\label{eq2.4}
	\Phi_i(t^+,t^-)=g_i^+(t^+)+g_i^-(t^-)
+J_\varepsilon(u_1,\dots,u_{i-1},0,u_{i+1},\dots,u_m),
\end{equation}
where, for $\sigma\in\{+,-\}$,
\begin{equation}\label{eq2.5}
	g_i^\sigma(t)
	= \frac12 a_i^\sigma t^2
	- \frac{\mu_i}{2p} b_i^\sigma t^{2p}
	- \frac{1}{2p}\sum_{j\neq i}\lambda_{ij} c_{ij}^\sigma t^p,~t>0,
\end{equation}
with coefficients
\[
a_i^\sigma:=\|u_i^\sigma\|_{H^1}^2,~
b_i^\sigma:=\int_{\mathbb{R}^N}Q_\varepsilon(x-y_i)|u_i^\sigma|^{2p}\,dx,
~
c_{ij}^\sigma:=\int_{\mathbb{R}^N}|u_j|^p|u_i^\sigma|^p\,dx.
\]
Obviously, $a_i^\sigma>0$, $b_i^\sigma\ge0$ and $c_{ij}^\sigma\ge0$.

\begin{Lem}\label{lem2.2}
	Let $\mathbf{u}\in H$ with $u_i^\pm\neq0$ for all $i=1,\dots,m$. Then for each $i$, there exists a
	unique pair $(t_i^+,t_i^-)\in(0,\infty)^2$ such that
	\[
	t_i^+u_i^+ + t_i^-u_i^- \in \mathcal{N}^{\mathrm{sc}}_{\varepsilon,i}.
	\]
Moreover, for $\sigma\in\{+,-\}$, $t_i^\sigma$ is the unique positive zero point of
\begin{equation}\label{eq2.6}
		f_i^\sigma(t)
		:= a_i^\sigma t^{2-p}
		- \mu_i b_i^\sigma t^{p}
		- \frac12 \sum_{j\neq i}\lambda_{ij} c_{ij}^\sigma,~ t>0.
\end{equation}
In particular, $\mathcal N_\varepsilon^{\mathrm{sc}}\neq\varnothing$.
\end{Lem}

\begin{proof}
	By \eqref{eq2.4}--\eqref{eq2.5}, the stationarity conditions
	$\partial_{t^\pm}\Phi_i(t^+,t^-)=0$ decouple into the two one-dimensional equations
	$(g_i^\sigma)'(t^\sigma)=0$, $\sigma\in\{+,-\}$. A direct computation yields
	\begin{equation}\label{eq2.7}
		(g_i^\sigma)^\prime(t)
		= a_i^\sigma t - \mu_i b_i^\sigma t^{2p-1}
		- \frac{1}{2}\sum_{j\neq i}\lambda_{ij} c_{ij}^\sigma t^{p-1},~ t>0.
	\end{equation}
	Since $t>0$, dividing \eqref{eq2.7} by $t^{p-1}$ gives exactly \eqref{eq2.6}.
	Thus critical points of $\Phi_i$ correspond to positive zeros of $f_i^\sigma$.

	Fix $i$ and $\sigma$ and consider $f:=f_i^\sigma$ on $(0,\infty)$. 
	Since $p>1$, we have $t^p \to 0$ as $t\to0^+ $. Using
	$a_i^\sigma>0$, $\lambda_{ij}<0$ and $c_{ij}^\sigma\ge0$, we obtain
	\[
	\liminf_{t\to 0} f(t)
	\;\ge\;
	-\frac12 \sum_{j\ne i}\lambda_{ij} c_{ij}^\sigma
	\;>\; 0.
	\]

	In particular, $f(t)>0$ for all $t>0$ sufficiently small.
Hence the continuous function  $f$ has at least one positive zero since $\lim\limits_{t\to\infty}f_i^\sigma(t)=-\infty$.
Direct calculations give
\begin{equation*}
  f^\prime(t)	= (2-p)a_i^\sigma t^{1-p} - \mu_i p b_i^\sigma t^{p-1}
\end{equation*}
and 
\begin{equation*}
  f^{\prime\prime}(t)
	= (2-p)(1-p)a_i^\sigma t^{-p} - \mu_i p(p-1)b_i^\sigma t^{p-2}.
\end{equation*}
	Since $p>1$ and $a_i^\sigma>0$, $b_i^\sigma\ge0$, we have $f''(t)<0$ for all $t>0$, so $f$ is strictly
	concave on $(0,\infty)$ and 
	has exactly one positive zero point $t_i^\sigma$. The pair $(t_i^+,t_i^-)$ is therefore uniquely
	determined and yields the desired element in $\mathcal{N}^{\mathrm{sc}}_{\varepsilon,i}$.
\end{proof}
\par
We now show that all components of an element of $\mathcal{N}^{\mathrm{sc}}_\varepsilon$
cannot be arbitrarily small, in the $H^{1}$-norm sense.

\begin{Lem}\label{lem2.3}
	If $\mathbf{u}\in\mathcal{N}^{\mathrm{sc}}_\varepsilon$, then $u_i^\pm\not\equiv0$ for all
	$i=1,\dots,m$, and there exists a constant $\delta>0$, independent of $\varepsilon$,
	such that
	\[
	\|u_i^\pm\|_{H^1}\;\ge\;\delta.
	\]
\end{Lem}

\begin{proof}
	For $\mathbf{u}\in\mathcal{N}^{\mathrm{sc}}_\varepsilon$, one has
	$u_i^\pm\not\equiv0$ for each $i$.  
	Suppose by contradiction that there exists 
	$\{\mathbf{u}^{(n)}\}\subset\mathcal{N}^{\mathrm{sc}}_\varepsilon$ and indices $i_n$ and signs
	$\sigma_n\in\{+,-\}$ such that $\|u^{(n)\,\sigma_n}_{i_n}\|_{H^1}\to0$.  
	It follows from \eqref{eq2.3} and the Sobolev inequality that
	\[
    \begin{aligned}
	\|u^{(n)\,\sigma_n}_{i_n}\|_{H^1}^2
	&=\mu_{i_n}\int_{\mathbb{R}^N}Q_\varepsilon|u^{(n)\,\sigma_n}_{i_n}|^{2p}\,dx
	+ \sum_{j\neq i_n}\lambda_{i_n j}
	\int_{\mathbb{R}^N}|u^{(n)}_j|^p|u^{(n)\,\sigma_n}_{i_n}|^p\,dx\\
    &\leq \mu_{i_n}\int_{\mathbb{R}^N}Q_\varepsilon|u^{(n)\,\sigma_n}_{i_n}|^{2p}\,dx\\
    & \le C\|u^{(n)\,\sigma_n}_{i_n}\|_{H^1}^{2p}
    \end{aligned}
	\]
since $\lambda_{ij}<0$, which 
	yields a uniform lower bound $\delta>0$ for all $\|u_i^\pm\|_{H^1}$.
\end{proof}
\par
To obtain sign-changing solutions with all components are sign-changing, we introduce the following subset of $H$
	\[
	\mathcal{A}
	:=\bigl\{\mathbf{u}\in H:\ u_i^\pm\neq 0,~\forall~i\in\{1,\dots,m\}\bigr\}.
	\]
\par
We claim that $\mathcal A$ is open in $H$.
Indeed, for $w\in H^1(\mathbb R^N)$, the truncation mappings $w\mapsto w^\pm$ are Lipschitz on $H^1(\mathbb R^N)$, namely
	\[
	\|w^\pm-z^\pm\|_{H^1}\le \|w-z\|_{H^1},~\forall\, w,z\in H^1(\mathbb R^N).
	\]
	Hence, for each fixed $i$ and $\sigma\in\{+,-\}$, the mapping
	$\mathbf u\mapsto u_i^\sigma$ is continuous from $H$ to $H^1(\mathbb R^N)$.
	Since $H^1(\mathbb R^N)\setminus\{0\}$ is open, it follows that
	\[
	\mathcal A=\bigcap_{i=1}^m\bigcap_{\sigma\in\{+,-\}}
	\{\mathbf u\in H|u_i^\sigma\in H^1(\mathbb R^N)\setminus\{0\}\}
	\]
	is open in $H$, since the right hand set is the intersection of finite open sets. Or equivalently, one can check $\mathcal{A}$ is open in a direct way.
Indeed, for $\mathbf{u}\in \mathbf{A}$, define 
$$r = \min\limits_{1\le i\le m,\sigma\in\{+,-\}}\|u_i^\sigma\|_{H^1}>0,$$
then for $\mathbf{v}\in B_{r/2}(\mathbf{u})$, there holds
\[
\|v_i^\sigma\|_{H^1} 
\geq \|u_i^\sigma\|_{H^1} - \|v_i^\sigma-u_i^\sigma\|_{H^1}
\geq \|u_i^\sigma\|_{H^1} - \|v_i - u_{i}\|_{H^1} 
> \delta - \frac{\delta}{2} = \frac{\delta}{2} > 0,
\]
which implies $v_i^\sigma\neq 0$ and hence $\mathbf{v}\in \mathcal{A}$. Therefore $\mathcal{A}$ is open. Consequently, $\mathcal{A}$ is a smooth Hilbert manifold modeled on $H$ and the tangent space $T_u\mathcal{A}=H$. Moreover, 
\begin{equation}\label{eq2.8}
  \partial\mathcal{A}=\{\mathbf{u}\in H:~\text{there exists}~i\in\{1,\cdots,m\}~\text{and}~\sigma\in\{+,-\}~\text{such that}~u_i^\sigma=0\}.
\end{equation}
Here we point out that $\mathcal{A}$ is not convex.
\par
We now introduce the nodal projection and the reduced functional. The uniqueness in
	Lemma~\ref{lem2.2} allows us to define $(t_i^+(\mathbf{u}),t_i^-(\mathbf{u}))$ for
	$\mathbf{u}\in\mathcal{A}$. 
\begin{Def}\label{def2.2}
	For $\mathbf{u}\in\mathcal{A}$, define $m_\varepsilon(\mathbf{u})\in \mathcal{N}^{\mathrm{sc}}_\varepsilon$ by
	\[
	\bigl(m_\varepsilon(\mathbf{u})\bigr)_i
	:= t_i^+(\mathbf{u})\,u_i^+ + t_i^-(\mathbf{u})\,u_i^-,~ i=1,\dots,m,
	\]
	and the reduced functional $\Psi_\varepsilon:\mathcal{A}\to\mathbb{R}$ by
	\begin{equation}\label{eq2.9}
		\Psi_\varepsilon(\mathbf{u})
		= J_\varepsilon\bigl(m_\varepsilon(\mathbf{u})\bigr).
	\end{equation}
\end{Def}
\begin{Rem}\label{rem2.2}
There are only few papers concerned with the open set constrain.
  In \cite{Buffoni-Esteban-Sere-2006-ANS}, to seek for normalized solutions for a strongly indefinite semilinear elliptic equation, Buffoni, Esteban and S\'{e}r\'{e} introduced a reduced functional on an open and convex set of $H^1(\mathbb{R}^N)$, such a topic related to the bifurcation in gaps of the essential spectrum of the corresponding differential operator.
\end{Rem}
\par
\begin{Lem}\label{lem2.4}
  The mapping $m_\varepsilon:\mathcal{A}\to \mathcal{N}^{\mathrm{sc}}_\varepsilon$ is an odd $C^1$-diffeomorphism.
\end{Lem}
\begin{proof}
Replacing $\mathbf{u}$ by $-\mathbf{u}$ swaps $u_i^+$ and $u_i^-$ but leaves
	\eqref{eq2.6} invariant, so $t_i^\pm(-\mathbf{u})=t_i^\mp(\mathbf{u})$, and hence $m_\varepsilon$ is odd.
\par
By the definition of $m_\varepsilon$, \( m_\varepsilon(\mathbf{u}) \) is a linear combination of \( t_i^+(\mathbf{u}) u_i^+ \) and \( t_i^-(\mathbf{u}) u_i^- \) for each \( i \). Since \( u_i^\sigma \) is Lipschitz continuous on \( H^1(\mathbb{R}^N) \), and hence we only need to show \( t_i^\sigma(\mathbf{u}) \in C^1(\mathcal{A}, (0, \infty)) \).
\par
For fixed \( i \) and \( \sigma \in \{+, -\} \), define
\[
F_i^\sigma(\mathbf{u}, t) = a_i^\sigma(\mathbf{u}) t^{2-p} - \mu_i b_i^\sigma(\mathbf{u}) t^p - \frac{1}{2} \sum_{j \neq i} \lambda_{ij} c_{ij}^\sigma(\mathbf{u}),~(\mathbf{u}, t) \in \mathcal{A} \times (0, \infty).
\]
By Lemma \ref{lem2.2}, \( F_i^\sigma(\mathbf{u}, t_i^\sigma(u)) = 0 \) for all \( \mathbf{u} \in \mathcal{A} \). Obviously, \( F_i^\sigma \in C^1(\mathcal{A} \times (0, \infty), \mathbb{R}) \) and 
\[
\frac{\partial F_i^\sigma}{\partial t}(\mathbf{u}, t)|_{(\mathbf{u},t_i^\sigma(\mathbf{u}))} = (2 - p) a_i^\sigma(\mathbf{u}) (t_i^\sigma(\mathbf{u}))^{1-p} - \mu_i p b_i^\sigma(\mathbf{u}) (t_i^\sigma(\mathbf{u}))^{p-1}<0
\]
since \( a_i^\sigma(\mathbf{u}) > 0 \), \( b_i^\sigma(\mathbf{u}) \geq 0 \) and \( p > 1 \).
By the Implicit Functiona Theorem, \( t_i^\sigma(\mathbf{u}) \) is \( C^1 \) in \( \mathbf{u} \in \mathcal{A} \). Since this holds for all \( i \) and \( \sigma \), \( m_\varepsilon(\mathbf{u}) \) is \( C^1 \) on \( \mathcal{A} \).
\par
Next, we show that \( m_\varepsilon: \mathcal{A} \to \mathcal{N}_\varepsilon^{sc} \) is bijective. Using Lemma \ref{lem2.2} again, $m_\varepsilon$ is one to one. 
For any \(\mathbf{v} \in \mathcal{N}_\varepsilon^{sc} \), \( v_i^+ \neq 0 \) and \( v_i^- \neq 0 \), so \( \mathbf{v} \in \mathcal{A} \) and $\langle J_\varepsilon^\prime(\mathbf{v}),v_i^+\rangle= 0$ and 
$\langle J_\varepsilon^\prime(\mathbf{v}),v_i^-\rangle = 0$, and the scaling factors for \( \mathbf{v} \) are \( t_i^+(\mathbf{v}) = 1 \) and \( t_i^-(\mathbf{v}) = 1 \). Thus, \( m_\varepsilon(\mathbf{v}) = 1 \cdot v_i^+ + 1 \cdot v_i^- = \mathbf{v} \). Therefore, $m_\varepsilon$ is surjective.
\par
The above fats implies that the inverse map \( m_\varepsilon^{-1}: \mathcal{N}_\varepsilon^{sc} \to \mathcal{A} \) is well defined. 
At last we show that $m_\varepsilon^{-1}$ is  \( C^1 \).
\par
Indeed, recall that \( \mathcal{A} \) is an open subset of \( H \), hence a Hilbert manifold modeled on \( H \) and the tangent space $T_u\mathcal{A}=H$. By the Regular Value Theorem, the constraints \( G_i^\sigma(w) = \langle J_\varepsilon'(w), w_i^\sigma \rangle = 0 \) (for \( i = 1, \dots, m \), \( \sigma \in \{+, -\} \)) are regular, so \( \mathcal{N}_\varepsilon^{sc} \) is a \( C^1 \)-submanifold of \( H \). Let \( T_{m_\varepsilon(u)}\mathcal{N}_\varepsilon^{sc} \) be the tangent space to \( \mathcal{N}_\varepsilon^{sc} \) at \( m_\varepsilon(u) \), which has codimension \(  2m \) since there are $2$ constraints per component.
The derivative \( m_\varepsilon^\prime(u) \) encodes the variation of scaling factors \( t_i^\sigma(u) \) and the original components \( u_i^\sigma \). By the non-degeneracy of \( \frac{\partial F_i^\sigma}{\partial t} \) and bijectivity of \( m_\varepsilon \), the Fr\'{e}chet derivative \( m_\varepsilon^\prime(u): T_u\mathcal{A} \to T_{m_\varepsilon(u)}\mathcal{N}_\varepsilon^{sc} \) is both injective and surjective, hence a linear isomorphism.  
\par
By the Inverse Function Theorem, \( m_\varepsilon \) is locally \( C^1 \)-invertible at every \( u \in \mathcal{A} \). Since \( m_\varepsilon \) is globally bijective, the inverse map \( m_\varepsilon^{-1} \) is globally \( C^1 \)-smooth on \( \mathcal{N}_\varepsilon^{sc} \).
\end{proof}

\begin{Lem}
	\label{lem2.5}
	The following statements hold.
	\begin{itemize}
		\item[(a)] The reduced functional
		$\Psi_\varepsilon\in C^1(\mathcal{A},\mathbb{R})$ is even.
		\item[(b)] A point $\mathbf{u}\in\mathcal{A}$ is a critical point of
		$\Psi_\varepsilon$ if and only if
		$\mathbf{v}:=m_\varepsilon(\mathbf{u})\in \mathcal{N}^{\mathrm{sc}}_\varepsilon$ is a
		critical point of $J_\varepsilon$ in $H$, i.e.\ $J_\varepsilon^\prime(\mathbf{v})=0$.
        \item[(c)]  
        If $(\mathbf{u}_n)\subset \mathcal{A}$ is a $(PS)_c$ sequence for
		$\Psi_\varepsilon$, namely
		\[
		\Psi_\varepsilon(\mathbf{u}_n)\to c
		~\text{and}~
		\Psi_\varepsilon^\prime(\mathbf{u}_n)\to 0 \ \text{in } H^\prime,
		\]
		then the sequence $\{\mathbf{v}_n\}$ is a $(PS)_c$ sequence for $J_\varepsilon$, where $\mathbf{v}_n := m_\varepsilon(\mathbf{u}_n)\in \mathcal{N}^{\mathrm{sc}}_\varepsilon$, 
		that is,
		\[
		J_\varepsilon(\mathbf{v}_n)\to c
		~\text{and}~
		J_\varepsilon'(\mathbf{v}_n)\to 0 \ \text{in } H'.
		\]
		\item[(d)] There exists a constant
		$\alpha>0$, independent of $\varepsilon$, such that for every
		$\mathbf{u}\in\mathcal{A}$ we have
		\[
		\Psi_{\varepsilon}(\mathbf{u})
		= J_{\varepsilon}\big(m_{\varepsilon}(\mathbf{u})\big)
		= \alpha\,\big\|m_{\varepsilon}(\mathbf{u})\big\|^{2}.
		\]
	\end{itemize}
\end{Lem}

\begin{proof}
	(a) Obviously, $\Psi_\varepsilon$ is even since $m_\varepsilon$ is odd.
By Lemma \ref{lem2.4}, $m_\varepsilon$ is $C^1$, and hence $\Psi_\varepsilon(u)=J_\varepsilon(m_\varepsilon(u))$ is $C^1$. 
\par
(b) Let $\mathbf{u}\in\mathcal{A}$ be a critical point of
	$\Psi_{\varepsilon}$ and set	$\mathbf{v}:=m_{\varepsilon}(\mathbf{u})\in\mathcal{N}_{\varepsilon}^{\mathrm{sc}}$.
	By the chain rule we have
	\[
	\langle\Psi_{\varepsilon}^\prime(\mathbf{u}),\boldsymbol{\varphi}\rangle
	=\langle J_{\varepsilon}^\prime(\mathbf{v}),  {m_{\varepsilon}}^\prime(\mathbf{u})\boldsymbol{\varphi}\rangle,~\boldsymbol{\varphi}\in T_{\mathbf{u}}\mathcal{A}.
	\]
Since $\Psi_{\varepsilon}^{\prime}(\mathbf{u})=0$ and
	$m_{\varepsilon}$ is a $C^1$-diffeomorphism between
	$\mathcal{A}$ and $\mathcal{N}_{\varepsilon}^{\mathrm{sc}}$, the range of
	${m_{\varepsilon}}^\prime(\mathbf{u})$ coincides with the tangent space
	$T_{\mathbf{v}}\mathcal{N}_{\varepsilon}^{\mathrm{sc}}$. Hence
	\[
	\langle J_{\varepsilon}^\prime(\mathbf{v}),\mathbf{w}\rangle=0,~\mathbf{w}\in T_{\mathbf{v}}\mathcal{N}_{\varepsilon}^{\mathrm{sc}},
	\]
	that is, $\mathbf{v}$ is a constrained critical point of $J_{\varepsilon}$ on
	$\mathcal{N}_{\varepsilon}^{\mathrm{sc}}$.
	
	Next we show that $\mathbf{v}$ is in fact an unconstrained critical point of
	$J_{\varepsilon}$. Recall that the sign-changing Nehari manifold can be written as
	\[
	\mathcal{N}_{\varepsilon}^{\mathrm{sc}}
	= \Bigl\{
	\mathbf{w}\in H : G_i^\sigma(\mathbf{w})=0,
	\forall i\in\{1,\dots,m\},~\sigma\in\{+,-\}
	\Bigr\},
	\]
	where
	\[
	G_i^\sigma(\mathbf{w})
	:= \bigl\langle J_{\varepsilon}^\prime(\mathbf{w}), w_i^\sigma\bigr\rangle,
~i=1,\dots,m,\ \sigma\in\{+,-\}.
	\]
\par	
	By the Lagrange multiplier rule, there exist real numbers
	$\lambda_i^\sigma$ such that
	\begin{equation}\label{eq2.10}
		J_{\varepsilon}'(\mathbf{v})
		= \sum_{i=1}^m\sum_{\sigma\in\{+,-\}}
		\lambda_i^\sigma\, G_i^\sigma{}'(\mathbf{v})~\text{in}~H^\prime.
	\end{equation}
	
	For each $j=1,\dots,m$ and $\tau\in\{+,-\}$, we evaluate
	\eqref{eq2.10} in the direction $v_j^\tau$. Since
	$\mathbf{v}\in\mathcal{N}_{\varepsilon}^{\mathrm{sc}}$, we have
	$\langle J_{\varepsilon}'(\mathbf{v}), v_j^\tau\rangle = 0$, and thus
	\[
	0
	= \bigl\langle J_{\varepsilon}'(\mathbf{v}), v_j^\tau\bigr\rangle
	= \sum_{i=1}^m\sum_{\sigma\in\{+,-\}}
	\lambda_i^\sigma\, \langle (G_i^\sigma)^\prime(\mathbf{v}),v_j^\tau\rangle.
	\]
	
	This yields a homogeneous linear system for the coefficients
	$\lambda_i^\sigma$ with coefficient matrix
	\[
	A_{(j,\tau),(i,\sigma)}
	:= \langle (G_i^\sigma)^\prime(\mathbf{v}),v_j^\tau\rangle,~i,j=1,\dots,m,\ \sigma,\tau\in\{+,-\}.
	\]
	On the other hand, by the definition of the fiber maps
	$g_i^\sigma$ in Lemma~\ref{lem2.2} and a straightforward computation using
	\eqref{eq1.1} and \eqref{eq2.1}, the matrix $A$ can be
	identified (up to positive factors) with the Hessian of the function
	\[
	(t_i^\sigma)_{i,\sigma}\longmapsto
	J_{\varepsilon}\Bigl(
	\sum_{i=1}^m \sum_{\sigma\in\{+,-\}} t_i^\sigma v_i^\sigma
	\Bigr)
	\]
	at the point corresponding to $\mathbf{v}$. By Lemma~\ref{lem2.2}, this Hessian is
	negative definite, since $\mathbf{v}$ is the unique maximizer of $J_{\varepsilon}$
	along its fiber. In particular, the matrix $A$ is invertible. Hence the
	only solution of the above linear system is
	$\lambda_i^\sigma=0$ for all $i$ and $\sigma$, and from \eqref{eq2.10} we
	conclude that
	\[
	J_{\varepsilon}'(\mathbf{v})=0~\text{in}~H^\prime.
	\]
	Thus $\mathbf{v}=m_{\varepsilon}(\mathbf{u})$ is an critical
	point of $J_{\varepsilon}$, and by construction, it is sign-changing.
	This proves part~(b).
	\par
(c) This is a direct consequence of $(b)$, see also \cite{ContiTerraciniVerzini2002}.
\par
(d) Let $\mathbf{v}:=m_{\varepsilon}(\mathbf{u})\in\mathcal{N}_{\varepsilon}^{\mathrm{sc}}$.
	By \eqref{eq2.3} and summing over
	$i=1,\dots,m$ and over $\sigma\in\{+,-\}$ we obtain
	\begin{equation}\label{eq2.11}
		\|\mathbf{v}\|^{2}
		= \sum_{i=1}^{m}\mu_{i}
		\int_{\mathbb{R}^{N}}Q_{\varepsilon}(x-y_{i})\,|v_{i}|^{2p}\,dx
		+ \sum_{i\neq j}\lambda_{ij}
		\int_{\mathbb{R}^{N}}|v_{i}|^{p}|v_{j}|^{p}\,dx .
	\end{equation}
Substituting \eqref{eq2.11} into \eqref{eq2.1} yields
	\begin{align*}
		J_{\varepsilon}(\mathbf{v})
		&= \frac{1}{2}\|\mathbf{v}\|^{2}
		- \frac{1}{2p}\sum_{i=1}^{m}\mu_{i}
		\int_{\mathbb{R}^{N}}Q_{\varepsilon}(x-y_{i})\,|v_{i}|^{2p}\,dx
		- \frac{1}{2p}\sum_{i\neq j}\lambda_{ij}
		\int_{\mathbb{R}^{N}}|v_{i}|^{p}|v_{j}|^p\,dx \\[2mm]
		&= \frac{1}{2}\|\mathbf{v}\|^{2}
		- \frac{1}{2p}\Biggl(
		\|\mathbf{v}\|^{2}
		- \sum_{i\neq j}\lambda_{ij}
		\int_{\mathbb{R}^{N}}|v_{i}|^{p}|v_{j}|^p\,dx
		\Biggr)
		- \frac{1}{2p}\sum_{i\neq j}\lambda_{ij}
		\int_{\mathbb{R}^{N}}|v_{i}|^{p}|v_{j}|^p\,dx \\[2mm]
		&= \Bigl(\frac{1}{2}-\frac{1}{2p}\Bigr)\|\mathbf{v}\|^{2}.
	\end{align*}
	
	Therefore, taking $\alpha=\frac{1}{2}-\frac{1}{2p}>0$ we obtain
	\[
	\Psi_{\varepsilon}(\mathbf{u})
	= J_{\varepsilon}(\mathbf{v})
	= \alpha\,\|\mathbf{v}\|^{2},
	\]
	which clearly implies the desired coercivity on $\mathcal{N}_{\varepsilon}^{\mathrm{sc}}$.
\end{proof}

\begin{Rem}\label{rem2.3}
	Lemma \ref{lem2.5} $(d)$ implies that $\Psi_\varepsilon(\mathbf{u})\to +\infty$ as $dist(\mathbf{u},\mathcal{A})\to 0$. In fact, if $dist(\mathbf{u},\mathcal{A})\to 0$, by \eqref{eq2.8}, there exist some $i\in\{1,\cdots,m\}$ and $\sigma\in\{+,-\}$ such that $\|u_i^\sigma\|_{H^1}\to 0$. By definition \ref{def2.2}, $\bigl(m_\varepsilon(\mathbf{u})\bigr)_i^\sigma
	:= t_i^\sigma(\mathbf{u})\,u_i^\sigma$. By Lemma \ref{lem2.3}, 
	\begin{equation*}		t_i^\sigma(\mathbf{u})=\frac{\bigl(m_\varepsilon(\mathbf{u})\bigr)_i^\sigma}{\|u_i^\sigma\|_{H^1}}
		\geq \frac{\delta}{\|u_i^\sigma\|_{H^1}}\to +\infty,
	\end{equation*}
	which implies that $\big\|m_{\varepsilon}(\mathbf{u})\big\|\to \infty$ as $dist(\mathbf{u},\mathcal{A})\to 0$. Thus 
	\begin{equation*}
		\Psi_{\varepsilon}(\mathbf{u})
		= J_{\varepsilon}\big(m_{\varepsilon}(\mathbf{u})\big)
		= \alpha\,\big\|m_{\varepsilon}(\mathbf{u})\big\|^{2}\to +\infty
	\end{equation*}
	as $dist(\mathbf{u},\mathcal{A})\to 0$. Roughly speaking, if $\Psi_\varepsilon$ has a critical point $\mathbf{u}\in \mathcal{A}$ with finite energy, $\mathbf{u}$ is far away from the boundary $\partial \mathcal{A}$. 
\end{Rem}

\section{The $(PS)_c$ condition of $\Psi_\varepsilon$}\label{sec:ps}
\par
In order to apply the minimax scheme to construct sign-changing solutions, we must establish the compactness property for the reduced functional
$\Psi_\varepsilon$.
The main difficulties arise from the lack of compactness of the Sobolev embedding  and the
presence of sign components. To obtain the compactness, we combine concentration compactness arguments with
the fully sign-changing Nehari structure introduced in Section~\ref{sec:framework}. As we will see below, the compactness for $(PS)$ sequences relies on the interplay between the potential $Q$ and the competitive couplings $\lambda_{ij}<0$.
\par
Recall a sequence $\{\mathbf{u}_n\}\subset\mathcal{A}$ is called a $(PS)_c$ sequence for $\Psi_\varepsilon$ if
\begin{equation*}
 \Psi_\varepsilon(\mathbf{u}_n)\to c~\text{and}~
 \Psi_\varepsilon^\prime(\mathbf{u}_n)\to 0~\text{in}~H^\prime.
\end{equation*}
The functional $\Psi_\varepsilon$ satisfies the $(PS)_c$ condition on $\mathcal A$ if every $(PS)_c$- sequence has a convergent subsequence in $\mathcal{A}$ and its limit also contained in $\mathcal{A}$.

\begin{Lem}
	\label{lem3.1}
	The reduced functional $\Psi_{\varepsilon}$ satisfies the
	$(PS)_c$ condition on $\mathcal A$ for $c\in\mathbb{R}$.
\end{Lem}

\begin{proof}
Let $\{\mathbf{u}_n\}\subset \mathcal{A}$ be a $(PS)_c$ sequence for $\Psi_\varepsilon$, i.e.
    \[
    \Psi_\varepsilon(\mathbf{u}_n)\to c~\text{and}~(\Psi_\varepsilon)'(\mathbf{u}_n)\to 0 \ \text{in } H^\prime.
    \]
    Set $\mathbf{v}_n:=m_\varepsilon(\mathbf{u}_n)\in \mathcal{N}^{\mathrm{sc}}_\varepsilon$.
    By Lemma~\ref{lem2.5}(c),
    \[
    J_\varepsilon(\mathbf{v}_n)=\Psi_\varepsilon(\mathbf{u}_n)\to c~\text{and}~
    J_\varepsilon'(\mathbf{v}_n)\to 0 \ \text{in } H^\prime.
    \]
In another word, $\{\mathbf{v}_n\}$ is a $(PS)_c$ sequence for $J_\varepsilon$ in $H$.
\par
Our main strategy is to prove that $\{\mathbf{v}_n\}$ has a strongly convergent subsequence in $H$ and the limit lies in $\mathcal{N}^{\mathrm{sc}}_\varepsilon\subset \mathcal{A}$. Jointly with Lemma~\ref{lem2.5}, $\Psi_\varepsilon$ satisfies the $(PS)_c$ condition.
\par
We first show that $\{\mathbf{v}_n\}$ is bounded in $H$.
In fact, using $J_\varepsilon(\mathbf{v}_n)\to c$ and $J_\varepsilon^\prime(\mathbf{v}_n)\to0$, we obtain
	\[
	\frac{p-1}{2p}\,\|\mathbf{v}_n\|^2
	= J_\varepsilon(\mathbf{v}_n) - \frac{1}{2p} \langle J_\varepsilon^\prime(\mathbf{v}_n),\mathbf{v}_n\rangle
	\le c + o(1) + o(1)\,\|\mathbf{v}_n\|,
	\]
which implies that 	
	$\{\mathbf{v}_n\}$ is bounded in $H$.
	Up to a subsequence, $\mathbf{v}_n\rightharpoonup \mathbf{v}$ in $H$ and
	$\mathbf{v}_n\to \mathbf{v}$ in $L^q_{\rm loc}(\mathbb R^N)^m$ for all $q\in[1,2^*)$.
Since $J_\varepsilon'(\mathbf{v}_n) \to 0$ and $\mathbf{v}_n \rightharpoonup \mathbf{v}$, 
using the fact that $\mathbf{v}_n \to \mathbf{v}$ in $L^{2p}_{\text{loc}}(\mathbb{R}^N)^m$ and $(A_3)$, one can check $\mathbf{v}$ is a critical point of $J_\varepsilon$ in a standard way (see \cite{Willem1996}).
\par	
Set $\mathbf{w}_n:=\mathbf{v}_n-\mathbf{v}$. Obviously, $\mathbf{w}_n$ is bounded in $H$. Next we will show that, up to a subsequence,  $\mathbf{w}_n\to \mathbf{0}$ in $H$.
	Assume by contradiction that $\|\mathbf{w}_n\|\nrightarrow 0$.
\par	
	If $\{\mathbf{w}_n\}$ is vanishing, then for every $R>0$,
	\[
	\sup_{y\in\mathbb{R}^N} \int_{B_R(y)} |\mathbf{w}_n(x)|^2\,dx \to 0 ~ \text{as } n\to\infty.
	\]
	By the vanishing Lemma (see \cite{Lions1984} or \cite{Willem1996}),
this would imply $\mathbf{w}_n\to0$ in $L^{2p}(\mathbb{R}^N)^m$.
\par
Because $J_\varepsilon'(\mathbf{v}_n) \to 0$ and $J_\varepsilon'(\mathbf{v}) = 0$, we have
\begin{equation}\label{eq3.1}
\begin{aligned}
o(1)&=\langle J_\varepsilon'(\mathbf{v}_n) - J_\varepsilon'(\mathbf{v}),\, \mathbf{w}_n\rangle\\
&= \|\mathbf{w}_n\|^2- \sum_{i=1}^m \mu_i \int_{\mathbb{R}^N} Q_\varepsilon(x-y_i)\bigl(|v_{n,i}|^{2p-2}v_{n,i} - |v_i|^{2p-2}v_i\bigr)w_{n,i}\,dx \\
&~ - \sum_{i\neq j} \lambda_{ij} \int_{\mathbb{R}^N} \bigl(|v_{n,j}|^p|v_{n,i}|^{p-2}v_{n,i} - |v_j|^p|v_i|^{p-2}v_i\bigr) w_{n,i}\,dx\\
&= \|\mathbf{w}_n\|^2+o(1),
\end{aligned}
\end{equation}
where we use the fact $|v_{n,i}|^{2p-2}v_{n,i} \to |v_i|^{2p-2}v_i$ in $L^{\frac{2p}{2p-1}}(\mathbb{R}^N)$, the H\"older inequality and $Q_\varepsilon \in L^\infty$. This yields that $\|\mathbf{w}_n\| \to 0$, contradicts  $\|\mathbf{w}_n\| \nrightarrow 0$.
Therefore, by \cite{Lions1984}, $\{\mathbf{w}_n\}$ is non-vanishing, i.e., there exist constants $R>0$, $\delta>0$ and a sequence $\{z_n\}\subset\mathbb{R}^N$ such that
	\begin{equation}\label{eq3.2}
		\int_{B_R(z_n)} |\mathbf{w}_n|^2\,dx \ge \delta~\text{for all}~n.
	\end{equation}
Observe that $|z_n|\to\infty$. Otherwise, up to a subsequence, $z_n\to z\in\mathbb{R}^N$.
	Then for $n$ large, $B_R(z_n)\subset B_{R+1}(z)$, and since $\mathbf{w}_n\to \mathbf{0}$ in
	$L^2(B_{R+1}(z))^m$, hence
	\[
	\int_{B_R(z_n)} |\mathbf{w}_n(x)|^2\,dx \le \int_{B_{R+1}(z)} |\mathbf{w}_n(x)|^2\,dx \to 0,
	\]
contradicting \eqref{eq3.2}. Thus, up to a subsequence, $|z_n|\to\infty$.

Define two translated sequences
\[
\widetilde{\mathbf{v}}_n(x):=\mathbf{v}_n(x+z_n)~\text{and}~ 
\widetilde{\mathbf{w}}_n(x):=\mathbf{w}_n(x+z_n),~x\in\mathbb{R}^N.
\]
Note that $\widetilde{\mathbf{v}}_n=\widetilde{\mathbf{w}}_n+\mathbf{v}(\,\cdot+z_n)$,
combined with the facts that $\mathbf{v}\in L^2(\mathbb{R}^N)^m$ and $|z_n|\to\infty$, one has
\[
\|\mathbf{v}(\,\cdot+z_n)\|_{L^2(B_R(0))}^2
=\int_{B_R(0)}|\mathbf{v}(x+z_n)|^2\,dx
=\int_{B_R(z_n)}|\mathbf{v}(x)|^2dx\to 0~(n\to\infty). 
\]
Thus for all sufficiently large $n$,
\[
\|\mathbf{v}(\,\cdot+z_n)\|_{L^2(B_R(0))}<\frac{\sqrt{\delta}}{2},
\]
which implies that for large $n$ there holds
\[
\sqrt{\delta}\le\|\widetilde{\mathbf{w}}_n\|_{L^2(B_R(0))}
\le\|\widetilde{\mathbf{v}}_n\|_{L^2(B_R(0))}+\|\mathbf{v}(\,\cdot+z_n)\|_{L^2(B_R(0))}
<\|\widetilde{\mathbf{v}}_n\|_{L^2(B_R(0))}+\frac{\sqrt{\delta}}{2}.
\]
Consequently,
\[
\|\widetilde{\mathbf{v}}_n\|_{L^2(B_R(0))}>\frac{\sqrt{\delta}}{2}~\text{for large}~n. 
\]
Obviously, $\{\widetilde{\mathbf{v}}_n\}$ is bounded in $H$. Going to a subsequence if necessary, we assume
\[
\widetilde{\mathbf{v}}_n\rightharpoonup\widetilde{\mathbf{v}}~\text{in}~H~\text{and}~
\widetilde{\mathbf{v}}_n\to\widetilde{\mathbf{v}}~\text{strongly in }L^2(B_R(0))^m. 
\]
Thus
\[
\|\widetilde{\mathbf{v}}\|_{L^2(B_R(0))}
=\lim_{n\to\infty}\|\widetilde{\mathbf{v}}_n\|_{L^2(B_R(0))}
\ge\frac{\sqrt{\delta}}{2}>0,
\]
which implies $\widetilde{\mathbf{v}}\neq\mathbf{0}$.

\par
	By \textnormal{$(A_3)$}, there exists $R_0>0$ with $\operatorname{supp}(Q)\subset B_{R_0}(0)$. Hence, for each $i=1,\dots,m$,
	\[
	Q_\varepsilon(x-y_i)=Q(\varepsilon x-y_i)=0
	~\text{whenever}~ |x|>\frac{R_0+\max_{1\le i\le m}|y_i|}{\varepsilon}=:R_\varepsilon.
	\]
	In particular, 
\begin{equation}\label{eq3.3}
  \operatorname{supp}\bigl(Q_\varepsilon(\cdot-y_i)\bigr)\subset B_{R_\varepsilon}(0),~ i=1,\dots,m.
\end{equation}

Fix $\varphi\in C_c^\infty(\mathbb{R}^N)$ with $\operatorname{supp}(\varphi)\subset B_{R_1}(0)$ and test $J_\varepsilon'(\mathbf{v}_n)\to0$ with
\[
\boldsymbol{\phi}_n:=\bigl(\varphi(\cdot-z_n) w_1,\dots,\varphi(\cdot-z_n) w_m\bigr)\in H .
\]
Therefore, by \eqref{eq2.2}, we have
\begin{align}\label{eq3.4}
	\langle J_\varepsilon^\prime(\mathbf{v}_n),\boldsymbol{\phi}_n\rangle
	&=\sum_{i=1}^m\int_{\mathbb{R}^N}\Bigl(\nabla v_{n,i}\cdot\nabla\bigl(\varphi(\cdot-z_n)w_i\bigr)
	+v_{n,i}\,\varphi(\cdot-z_n)w_i\Bigr)\,dx \nonumber\\	&~-\sum_{i=1}^m\mu_i\int_{\mathbb{R}^N}Q_\varepsilon(x-y_i)\,|v_{n,i}|^{2p-2}v_{n,i}\,
	\varphi(x-z_n)w_i(x)\,dx \nonumber\\
	&~-\sum_{i\neq j}\lambda_{ij}\int_{\mathbb{R}^N}|v_{n,j}|^{p}|v_{n,i}|^{p-2}v_{n,i}\,
	\varphi(x-z_n)w_i(x)\,dx \;=\;o(1).
\end{align}

Since $|z_n|\to\infty$, for sufficiently large $n$ we have $|z_n|>R_1+R_\varepsilon$.
Then for any $x\in B_{R_1}(z_n)$,
\[
|x|\ge |z_n|-|x-z_n|> (R_1+R_\varepsilon)-R_1=R_\varepsilon .
\]
Hence, by \eqref{eq3.3} we obtain $Q_\varepsilon(x-y_i)=0$ for all $i=1,\dots,m$ and all such $x$.
Therefore,
\[
Q_\varepsilon(\cdot-y_i)\,\varphi(\cdot-z_n)\equiv0 \qquad \text{for all }i=1,\dots,m
\]
for $n$ sufficiently large.
Consequently, for $n$ large \eqref{eq3.4} reduces to
\begin{align}\label{eq3.5}
	&\sum_{i=1}^m\int_{\mathbb{R}^N}\Bigl(\nabla v_{n,i}\cdot\nabla\bigl(\varphi(\cdot-z_n)w_i\bigr)
	+v_{n,i}\,\varphi(\cdot-z_n)w_i\Bigr)\,dx \nonumber\\
	&~-\sum_{i\neq j}\lambda_{ij}\int_{\mathbb{R}^N}|v_{n,j}|^{p}|v_{n,i}|^{p-2}v_{n,i}\,
	\varphi(x-z_n)w_i(x)\,dx=o(1).
\end{align}
Changing variables $x\mapsto x+z_n$ in \eqref{eq3.5}, we obtain
\begin{equation}\label{eq:shifted-test}
	\sum_{i=1}^m\int_{\mathbb{R}^N}\Bigl(\nabla \tilde v_{n,i}\cdot\nabla(\varphi w_i)
	+\tilde v_{n,i}\,\varphi w_i\Bigr)\,dx \nonumber\\
	-\sum_{i\neq j}\lambda_{ij}\int_{\mathbb{R}^N}|\tilde v_{n,j}|^{p}|\tilde v_{n,i}|^{p-2}\tilde v_{n,i}\,
	\varphi w_i\,dx=o(1),
\end{equation}
which implies that $\widetilde{\mathbf{v}}$ is a weak solution of the limit system by letting $n\to \infty$
\begin{equation}\label{eq3.6}
	-\Delta \widetilde{v}_i+\widetilde{v}_i=\sum_{j\neq i}\lambda_{ij}|\widetilde{v}_j|^{p}|\widetilde{v}_i|^{p-2}\widetilde{v}_i
	~\text{in}~\mathbb{R}^N,~i=1,\dots,m.
\end{equation}
	Testing \eqref{eq3.6} by $\widetilde{v}_i$ and summing over $i$, we obtain
	\[
	\sum_{i=1}^m\int_{\mathbb R^N}\bigl(|\nabla \widetilde{v}_i|^2+|\widetilde{v}_i|^2\bigr)\,dx
	=
	\sum_{i\neq j}\lambda_{ij}\int_{\mathbb R^N}|\widetilde{v}_i|^{p}|\widetilde{v}_j|^{p}\,dx \le 0,
	\]
	because $\lambda_{ij}<0$ by \textnormal{$(A_2)$}. Therefore $\widetilde{\mathbf{v}}=\mathbf{0}$, and this is absurd. This contradiction implies that 
$\mathbf{v}_n\to\mathbf{v}$ in $H$.
\par
 Since $\mathbf{v}_n\in\mathcal{N}^{\mathrm{sc}}_\varepsilon$, Lemma \ref{lem2.3} gives a uniform lower bound
    $\|v_{n,i}^\pm\|_{H^1}\ge\delta>0$ for each $i$. This yields that
    $\|v_i^\pm\|_{H^1}\ge\delta>0$, hence $v_i^\pm\neq 0$ for all $i$.
    Moreover, $J_\varepsilon'(\mathbf{v})=0$.
    Testing this with $v_i^\pm$ yields $\langle J_\varepsilon^\prime(\mathbf{v}),v_i^\pm\rangle=0$, so $\mathbf{v}\in\mathcal{N}^{\mathrm{sc}}_\varepsilon\subset\mathcal{A}$.
\par
    Finally, by Lemma \ref{lem2.4},  $m_\varepsilon:\mathcal{A}\to\mathcal{N}^{\mathrm{sc}}_\varepsilon$ is a $C^1$-diffeomorphism, we have $\mathbf{u}_n=m_\varepsilon^{-1}(\mathbf{v}_n)\to m_\varepsilon^{-1}(\mathbf{v})$ strongly in $H$,
    and $m_\varepsilon^{-1}(\mathbf{v})\in\mathcal{A}$. Therefore $\{\mathbf{u}_n\}$ has a convergent subsequence
    in $\mathcal{A}$, proving that $\Psi_\varepsilon$ satisfies the $(PS)_c$ condition.   
\end{proof}

\begin{Rem}\label{rem3.1}
The shrinking potential $Q_\varepsilon$ localizes possible concentration to neighbourhoods of $y_i$, while the competitive couplings prevent energy from escaping to infinity. 
	This is one of the key mechanisms leading to compactness in the nodal setting.
\end{Rem}

\section{Existence of infinitely many sign-changing solutions}\label{sec:existence}
In this section we prove Theorem~\ref{thm1.1}. By Lemma~\ref{lem2.5}(a), the reduced functional
$\Psi_\varepsilon$ is an even $C^1$ functional on the symmetric set
$\mathcal{A}\subset H$. 
Our argument follows a symmetric minimax scheme for the reduced nodal functional
$\Psi_{\varepsilon}$ based on the Krasnosel'skii genus.
The crucial compactness is provided by Lemma \ref{lem3.1},
which ensures that $\Psi_{\varepsilon}$ satisfies the $(PS)_C$ condition
at every level $c$ for each fixed $\varepsilon>0$.
With this at hand, the genus minimax levels yield infinitely many critical points of
$\Psi_{\varepsilon}$, and via the nodal projection $m_{\varepsilon}$,
we obtain infinitely many nodal solutions of \eqref{eq1.1}.
\par
First, we recall the standard Krasnosel'skii genus of symmetrical sets.
\begin{Def}[\cite{Struwe1996}]\label{def4.1}
	A subset $E\subset\mathcal{A}$ is called \emph{symmetric} if $\mathbf{u}\in E$ implies $-\mathbf{u}\in E$.
	The Krasnosel'skii genus $\gamma(E)$ is defined as the least integer $k\in\mathbb{N}$
	such that there exists a continuous odd map $E\to\mathbb{R}^k\setminus\{0\}$.
	If no such $k$ exists, we set $\gamma(E)=\infty$.
\end{Def}
\par
{\bf Proof of Theorem \ref{thm1.1}.}
Fix $k\in\mathbb N^{*}$, we construct a symmetric finite dimensional subspace. Choose functions
\[
\boldsymbol\varphi_{1},\dots,\boldsymbol\varphi_{k}\in C_{c}^{\infty}(\mathbb{R}^{N})^m
\]
with pairwise disjoint supports and such that each component $\varphi_{j,i}$ changes sign in $\boldsymbol\varphi_{j}$'s support for $i=1,\cdots,m$, so that \(\varphi_{j,i}^{+}\) and \(\varphi_{j,i}^{-}\) are both nontrivial for $j=1,\cdots,k$ and $i=1,\cdots,m$.
Let
\[	
E_{k}:=\operatorname{span}\{\boldsymbol\varphi_{1},\dots,\boldsymbol\varphi_{k}\}
\]
and consider the sphere
\[
S_{k}:=\{\,\mathbf{u}\in E_{k}:\|\mathbf{u}\|=1\,\}.
\]
Since \(E_{k}\) is finite dimensional, \(S_{k}\) is compact and symmetric. 
For each $\sigma=(\sigma_1,\dots,\sigma_k)\in\{\pm1\}^{k}$, set
\[
\mathbf{u}_\sigma:=\sum_{j=1}^{k}\sigma_j\,\boldsymbol{\varphi}_j\in H.
\]
 $\mathbf{u}_\sigma\neq 0$ for every $\sigma\in\{\pm1\}^{k}$.
Moreover, the mapping
\[
\boldsymbol\sigma\longmapsto \frac{\mathbf{u}_{\boldsymbol\sigma}}{\|\mathbf{u}_{\boldsymbol\sigma}\|}
\]
induces an odd homeomorphism between the discrete set \(\{\pm1\}^{k}\) and a subset of \(S_{k}\) and hence $\gamma(S_{k})\ge k$. 
\par	
By construction every nonzero \(\mathbf{u}\in E_{k}\) has all components sign‑changing on their supports, hence $S_{k}\subset\mathcal{A}$.	
\par 
Set
\[
B_{k}:=m_{\varepsilon}(S_{k})\subset\mathcal{N}_{\varepsilon}^{\mathrm{sc}} .
\]
Because $m_{\varepsilon}$ is odd and continuous, the monotonicity of the genus under odd maps yields
\[
\gamma(B_{k})\;\ge\;\gamma(S_{k})\;\ge\;k .
\]
Denote
\begin{equation}\label{eq4.1}
  \Sigma_{k}=\{A\subset\mathcal{A}: A~\text{is compact and symmetric with}~\gamma(A)\ge k\},
\end{equation}
and define the minimax level
\begin{equation}\label{eq4.2}
	c_{\varepsilon,k}:=\inf_{A\in\Sigma_{k}}\;\sup_{\mathbf{u}\in A}\Psi_{\varepsilon}(\mathbf{u}). 
\end{equation}
Since $B_{k}\in\Sigma_{k}$ and $\Psi_{\varepsilon}$ is continuous and bounded from below on $B_{k}$, so each $c_{\varepsilon,k}$ is finite.
\par
Next, we claim that $c_{\varepsilon,k}$ is a critical value of $\Psi_{\varepsilon}$.
\par
In fact, by Lemma~\ref{lem3.1}, $\Psi_{\varepsilon}$ satisfies the $(PS)_c$ condition
for every $c\in\mathbb{R}$ on the open set $\mathcal A\subset H$.
Hence the symmetric minimax theorem in terms of the Krasnosel'skii genus (e.g. \cite[Theorem~5.7]{Struwe1996}) with modifications can be
applied. For completeness we give a
deformation argument on the open constraint set $\mathcal A$.

Assume by contradiction that $c_{\varepsilon,k}$ is not a critical value of $\Psi_{\varepsilon}$.
Then there exists $\delta>0$ such that $\Psi_{\varepsilon}$ has no critical values in
$[c_{\varepsilon,k}-2\delta,\,c_{\varepsilon,k}+2\delta]$.
Moreover, there exists $\sigma>0$ such that
\begin{equation}\label{eq4.3}
	\bigl\|\Psi_{\varepsilon}^\prime(\mathbf{u})\bigr\|_{H^\prime}\ge \sigma	,~\forall\mathbf{u}\in\Psi_\varepsilon^{-1}([c_{\varepsilon,k}-2\delta,\,c_{\varepsilon,k}+2\delta])
=\{\mathbf{u}\in \mathcal A:~\Psi_{\varepsilon}(\mathbf{u})\in I\}.
\end{equation}
Indeed, if \eqref{eq4.3} is false, we can find a $(PS)_c$-sequence $(\mathbf{u}_n)\subset\mathcal A$ of $\Psi_\varepsilon$ for some $c\in I$.  
By Lemma~\ref{lem3.1}, up to a subsequence, $\mathbf{u}_n\to \mathbf{u}$ in $H$ and $\mathbf{u}\in\mathcal A$.
Therefore $\Psi_{\varepsilon}'(\mathbf{u})=0$ and $\Psi_{\varepsilon}(\mathbf{u})=c\in I$,
which contradicts the choice of $I$.
Thus \eqref{eq4.3} holds.
\par
 Next we will apply a variant of the quantitative deformation lemma \cite[Lemma~5.15]{Willem1996} on $\mathcal A$. Define the set $\mathcal{S} := \Psi_\varepsilon^{-1}([c_{\varepsilon,k}-2\delta,\,c_{\varepsilon,k}+2\delta])$. In a standard way (\cite{Willem1996}), one can construct a locally Lipschitz continuous pseudo-gradient vector field $V: \mathcal{A} \to H$ satisfying
 \begin{equation}\label{eq4.4}
   \langle \Psi_{\varepsilon}(\mathbf{u}),V(\mathbf{u})\rangle \ge \frac{1}{2} \|\Psi_\varepsilon^\prime(\mathbf{u})\|_{H^\prime}^2,~
   \|V(\mathbf{u})\| \le 2 \|\Psi_{\varepsilon}^\prime(\mathbf{u})\|_{H^\prime},~\mathbf{u}\in \mathcal{S},
 \end{equation}
$ V(-\mathbf{u}) = -V(\mathbf{u})$ for $\mathbf{u}\in \mathcal{A}$, and $V(\mathbf{u})=0$ for $\mathbf{u}\in\mathcal A\setminus \mathcal{S}$.
\par
Choose an even Lipschitz cut-off function $\chi: \mathcal{A} \to [0,1]$ such that
\[
\chi(\mathbf{u}) = 
\begin{cases}
1, & \mathbf{u}\in\Psi_\varepsilon^{-1}([c_{\varepsilon,k} - \delta, c_{\varepsilon,k} + \delta]), \\
0, & \mathbf{u}\in \mathcal{A}\setminus \mathcal{S}.
\end{cases}
\]

\par
Consider the initial value problem
\begin{equation}\label{eq4.5}
\begin{cases}
\displaystyle \frac{d}{dt} \eta(t, \mathbf{u}) = -\chi(\eta(t, \mathbf{u}))V(\eta(t, \mathbf{u})), \\[6pt]
\eta(0, \mathbf{u}) = \mathbf{u}.
\end{cases}
\end{equation}
Since the right-hand side of \eqref{eq4.5} is locally Lipschitz on $\mathcal A$,  
$\chi V$ vanishes near $\partial\mathcal A$, and Remark \ref{rem2.3} implies that $\Psi_\varepsilon\to+\infty$ near $\partial\mathcal A$, the solution $\eta(t,\mathbf{u})$ 
exists for all $t\ge 0$ and stays in $\mathcal A$.  From the symmetry $V(-\mathbf{u}) = -V(\mathbf{u})$ and $\chi(-\mathbf{u}) = \chi(\mathbf{u})$, one has $\eta(t, -\mathbf{u}) = -\eta(t, \mathbf{u})$.
\par
Along the flow lines, by \eqref{eq4.4} and \eqref{eq4.5}, we have
\begin{equation*}
\begin{aligned}
\frac{d}{dt} \Psi_{\varepsilon}(\eta(t, \mathbf{u})) 
&= \langle \Psi_{\varepsilon}^\prime (\eta(t, \mathbf{u})),\frac{d}{dt} \eta(t, \mathbf{u})\rangle \\
&= -\chi(\eta) \langle\Psi_{\varepsilon}^\prime (\eta(t, \mathbf{u})),V(\eta(t, \mathbf{u}))\rangle\\ 
&\le -\frac{1}{2} \chi(\eta) \|\Psi_{\varepsilon}^\prime (\eta(t, \mathbf{u}))\|_{H^\prime}^2\\ 
&\le -\frac{\sigma^2}{2} \chi(\eta).
\end{aligned}
\end{equation*}
In particular, if $\Psi_{\varepsilon}(\mathbf{u}) \in [c_{\varepsilon,k} - \delta, c_{\varepsilon,k} + \delta]$, then $\chi(\mathbf{u}) = 1$ and thus
\[
\Psi_{\varepsilon}(\eta(t, \mathbf{u})) \le \Psi_{\varepsilon}(\mathbf{u}) - \frac{\sigma^2}{2} t.
\]
Set $T := \frac{4\delta}{\sigma^2}$.
Then for any $\mathbf{u} \in \Psi_{\varepsilon}^{-1}([c_{\varepsilon,k} - \delta, c_{\varepsilon,k} + \delta])$, we have
\[
\Psi_{\varepsilon}(\eta(T, \mathbf{u})) \le c_{\varepsilon,k} + \delta - \frac{\sigma^2}{2} \cdot T = c_{\varepsilon,k} - \delta.
\]
Define $\varphi: \mathcal{A} \to \mathcal{A}$ by
\[
\varphi(\mathbf{u}) := \eta(T, \mathbf{u}).
\]
Then $\varphi$ satisfies
\begin{align*}
	&\text{(a)}~\varphi(\mathbf{u})=\mathbf{u}~\text{if}~t=0~\text{or}~ 
	\mathbf{u}\notin \Psi_{\varepsilon}^{-1}\bigl([c_{\varepsilon,k}-2\delta,c_{\varepsilon,k}+2\delta]\bigr),\\
	&\text{(b)}~\varphi(\{\mathbf{u}\in \mathcal{A}:\Psi_{\varepsilon}(\mathbf{u})\le c_{\varepsilon,k}+\delta\})
	\subset \{\mathbf{u}\in \mathcal{A}:\Psi_{\varepsilon}(\mathbf{u})\le c_{\varepsilon,k}-\delta\},\\
    &\text{(c)}~\varphi~\text{is continuous and odd}.
\end{align*}
\par
By the definition of $c_{\varepsilon,k}$, choose $A\in\Sigma_k$ such that
$\sup_{\mathbf{u}\in A}\Psi_{\varepsilon}(\mathbf{u})\le c_{\varepsilon,k}+\delta$.
Then by (b),
\[
\sup_{\mathbf{u}\in \varphi(A)}\Psi_{\varepsilon}(\mathbf{u})\le c_{\varepsilon,k}-\delta.
\]
By (c),
$\gamma(\varphi(A))\ge \gamma(A)\ge k$, so $\varphi(A)\in\Sigma_k$.
Therefore,
\[
c_{\varepsilon,k}\le \sup_{\mathbf{u}\in \varphi(A)}\Psi_{\varepsilon}(\mathbf{u})\le c_{\varepsilon,k}-\delta,
\]
a contradiction. Hence $c_{\varepsilon,k}$ is a critical value of $\Psi_{\varepsilon}$.
\par
Finally, we show that $c_{\varepsilon,k}\to+\infty$ as $k\to\infty$.
This is a variant of a standard argument for unboundedness of genus minimax levels
(see, e.g., \cite[Proposition~9.33]{Rabinowitz1986}).
Assume by contradiction that $\{c_{\varepsilon,k}\}$ is bounded above and set
\[
\bar c_\varepsilon:=\sup_{k\ge1}c_{\varepsilon,k}<+\infty .
\]
Since $\Sigma_{k+1}\subset\Sigma_k$, the sequence $\{c_{\varepsilon,k}\}$ is nondecreasing and hence
$c_{\varepsilon,k}\uparrow\bar{c}_\varepsilon$ as $k\to\infty$.
If $c_{\varepsilon,k}=\bar{c}_\varepsilon$ for all $k$ large enough, then $\gamma(\mathcal{K}_{\bar{c}_\varepsilon})=\infty$ (see \cite{Rabinowitz1986}),
where 
\begin{equation*}
\mathcal{K}_{\bar{c}_\varepsilon}=\{\mathbf{u}\in \mathcal{A}:~\Psi_\varepsilon^\prime(\mathbf{u})=0
~\text{and}~\Psi_\varepsilon(\mathbf{u})=\bar{c}_\varepsilon\}.
\end{equation*}
But by Lemma \ref{lem3.1}, $\mathcal{K}_{\bar{c}_\varepsilon}$ is compact and so $\gamma(\mathcal{K}_{\bar{c}_\varepsilon})<\infty$. This contradiction implies that 
$c_{\varepsilon,k}<\bar{c}_\varepsilon$ for all $k$.
\par
Fix $\delta\in(0,1)$ and choose $k_0\ge1$ such that $c_{\varepsilon,k_0}>\bar{c}_\varepsilon-\delta$.
Consider the symmetric critical set
\[
\mathcal{K}:=\Bigl\{\mathbf{u}\in\mathcal A:\ \Psi_{\varepsilon}'(\mathbf{u})=\mathbf{0},\
\Psi_{\varepsilon}(\mathbf{u})\in[\bar{c}_\varepsilon-\delta,\bar{c}_\varepsilon]\Bigr\}.
\]
Using Lemma~\ref{lem3.1} again, $\mathcal{K}$ is compact in $H$. Moreover, $\mathcal{K}$ is symmetric since $\Psi_{\varepsilon}$ is even.
Hence there exists an open symmetric neighborhood $\mathcal O\subset\mathcal A$ of $\mathcal{K}$
such that $q:=\gamma(\mathcal O)<\infty$ (see \cite{Rabinowitz1986}).
Then, arguing as in \eqref{eq4.3}, there exists $\sigma_\delta>0$ such that
\[
\|\Psi_\varepsilon^\prime(u)\|_{H^\prime}\ge \sigma_\delta
~\text{for all}~\mathbf{u}\in \Psi_\varepsilon^{-1}([\bar{c}_\varepsilon-\delta,\bar{c}_\varepsilon+\delta])\setminus\mathcal O.
\]
Therefore one can repeat the deformation construction above on the open set $\mathcal A$
(based on \eqref{eq4.3}--\eqref{eq4.5}) with the cut-off supported in
$\Psi_{\varepsilon}^{\,\bar{c}_\varepsilon+\delta}\setminus\mathcal O$ and obtain an odd continuous map
$\varphi:\mathcal A\to\mathcal A$ such that
\begin{equation}\label{eq4.6}
	\varphi\bigl(\Psi_{\varepsilon}^{\,\bar{c}_\varepsilon+\delta}\setminus\mathcal O\bigr)\subset \Psi_{\varepsilon}^{\,\bar{c}_\varepsilon-\delta}.
\end{equation}
Next, by the definition of $c_{\varepsilon,k_0+q}$, choose $A\in\Sigma_{k_0+q}$ satisfying
\[
\sup_{\mathbf{u}\in A}\Psi_{\varepsilon}(\mathbf{u})< c_{\varepsilon,k_0+q}+\delta< \bar{c}_\varepsilon+\delta,
\]
so $A\subset \Psi_{\varepsilon}^{\,\bar{c}_\varepsilon+\delta}$.
Set $A_1:=A\setminus\mathcal O$. Then $A_1$ is compact and symmetric. By the subadditivity of the genus and the choice $q=\gamma(\mathcal O)$, we have
\[
k_0+q\le \gamma(A)\le \gamma(A\cap\mathcal O)+\gamma(A_1)\le q+\gamma(A_1),
\]
which implies $\gamma(A_1)\ge k_0$, i.e.\ $A_1\in\Sigma_{k_0}$.
Since $\varphi$ is odd and continuous, the genus monotonicity yields
$\gamma(\varphi(A_1))\ge\gamma(A_1)\ge k_0$, and thus $\varphi(A_1)\in\Sigma_{k_0}$. On the other hand, by \eqref{eq4.6}, we have $\varphi(A_1)\subset\Psi_{\varepsilon}^{\,\bar{c}_\varepsilon-\delta}$, which implies
\[
c_{\varepsilon,k_0}\le \sup_{\mathbf{u}\in\varphi(A_1)}\Psi_{\varepsilon}(\mathbf{u})\le \bar{c}_\varepsilon-\delta,
\]
contradicting $c_{\varepsilon,k_0}>\bar{c}_\varepsilon-\delta$.
This contradiction shows that $\{c_{\varepsilon,k}\}$ is unbounded, and therefore
$c_{\varepsilon,k}\to+\infty$ as $k\to\infty$.
\par

Consequently, for each $k\in\mathbb{N}^{*}$ there exists $\mathbf{v}_\varepsilon^{(k)}\in\mathcal{A}$ such that
\[
\Psi_{\varepsilon}\bigl(\mathbf{v}_\varepsilon^{(k)}\bigr)=c_{\varepsilon,k}
~\text{and}~
\Psi_{\varepsilon}^\prime\bigl(\mathbf{v}_\varepsilon^{(k)}\bigr)=0 .
\]

Set
\[
\mathbf{u}_\varepsilon^{(k)}:=m_{\varepsilon}\bigl(\mathbf{v}_\varepsilon^{(k)}\bigr)\in\mathcal{N}_{\varepsilon}^{\mathrm{sc}} .
\]
By Lemma~\ref{lem2.5}(b), $\mathbf{u}_\varepsilon^{(k)}$ is a critical point of $J_{\varepsilon}$ in $H$,
hence a weak solution of system \eqref{eq1.1}. Moreover,
\[
J_\varepsilon(\mathbf{u}_\varepsilon^{(k)})
=J_\varepsilon\!\bigl(m_\varepsilon(\mathbf{u}^{(k)})\bigr)
=\Psi_\varepsilon(\mathbf{u}^{(k)})=c_{\varepsilon,k}.
\]

Since $c_{\varepsilon,k}\to+\infty$ as $k\to\infty$, after relabeling, we may assume the energies are strictly increasing:
\[
0<J_\varepsilon\bigl(\mathbf{u}_\varepsilon^{(1)}\bigr)
<J_\varepsilon\bigl(\mathbf{u}_\varepsilon^{(2)}\bigr)
<\cdots<J_\varepsilon\bigl(\mathbf{u}_\varepsilon^{(k)}\bigr)
<\cdots,
~
\lim_{k\to\infty}J_\varepsilon\bigl(\mathbf{u}_\varepsilon^{(k)}\bigr)=+\infty.
\]
This completes the proof of Theorem~\ref{thm1.1}.

\begin{Rem}\label{rem4.1}
	When $m=1$, the fully sign-changing Nehari set $\mathcal{N}_{\varepsilon}^{\mathrm{sc}}$
	reduces to the classical nodal Nehari set for a single Schr\"odinger equation with a shrinking region
	of attraction. In this scalar case, our variational construction is consistent with the known results
	on the existence a least-energy nodal solution for each fixed
	$\varepsilon>0$; (see \cite{ClappHernandezSantamariaSaldana2025}).
	In this sense, Theorem~\ref{thm1.1} genuinely extends the scalar theory to competitive
	systems with $m\ge2$ components, where all components are required to be sign-changing.
\end{Rem}

\section{Concentration of sign-changing solutions as $\varepsilon\to 0^+$}
\label{sec:concentration}

In this section we prove Theorem~\ref{thm1.2}, which describes the concentration behavior
of sign-changing solutions obtained in Theorem~\ref{thm1.1} and the structure of their limiting profiles as $\varepsilon \to 0$.
\par
First, we recall the limiting autonomous system~\eqref{eq1.2}
\begin{equation}\notag
	\begin{cases}
		-\Delta U_i + U_i
		= \mu_i Q(-y_i)\,|U_i|^{2p-2}U_i
		+ \displaystyle\sum_{j\neq i}\lambda_{ij}\,|U_j|^p\,|U_i|^{p-2}U_i,
		\\[2mm]
		U_i \in H^1(\mathbb{R}^N), ~ i=1,\dots,m.
	\end{cases}
\end{equation}
The corresponding energy functional associated with \eqref{eq1.2} is
\begin{equation*}
	J_0(\mathbf{U})
	= \frac12\sum_{i=1}^m\int_{\mathbb{R}^N}\bigl(|\nabla U_i|^2+|U_i|^2\bigr)\,dx
	-\frac1{2p}\sum_{i=1}^m \mu_i Q(-y_i)\int_{\mathbb{R}^N}|U_i|^{2p}\,dx
	-\frac1{2p}\sum_{i\neq j}\lambda_{ij}\int_{\mathbb{R}^N}|U_i|^p|U_j|^p\,dx,
\end{equation*}
and the corresponding Nehari manifold is
\begin{equation*}
	\mathcal{N}_0
	:= \bigl\{\mathbf{U}\in H\setminus\{\mathbf{0}\}:\ \langle J_0^\prime(\mathbf{U}),\mathbf{U}\rangle=0\bigr\}.
\end{equation*}
Denote $\alpha:=\inf\limits_{\mathbf{U}\in\mathcal{N}_0}J_0(\mathbf{U})$, then $\alpha>0$ (see \cite{ClappSaldanaSzulkin2025}).
All definitions and properties introduced in Section~\ref{sec:framework} for $J_\varepsilon$ and $\Psi_\varepsilon$
can be repeated for $\varepsilon=0$. In particular, the fully sign-changing Nehari set is denoted by
$\mathcal{N}^{\mathrm{sc}}_0$, the reduced functional $\Psi_0:\mathcal{A}\to \mathbb{R}$ is defined by
\[
\Psi_0(\mathbf{u})
= J_0\bigl(m_0(\mathbf{u})\bigr),~ \mathbf{u}\in\mathcal{A},
\]
where $m_0:\mathcal{A}\to \mathcal{N}^{\mathrm{sc}}_0$ is the corresponding nodal projection.
\par
However, the functional $\Psi_0$ does not satisfy the $(PS)_c$ condition since it is invariant under translations. Therefore, we establish below the $(PS)_c$ property
modulo $\mathbb Z^N$-translations, which is sufficient for the genus minimax argument.
\par
For every $a\in\mathbb R$, set 
\begin{equation*}
  \mathcal{K}^a:=\{\mathbf u\in \mathcal A:~\Psi_0^\prime(\mathbf u)=0~\text{and}~\Psi_0(\mathbf u)\le a\}.
\end{equation*}
\begin{Lem}\label{Lem5.1}
The set of critical orbits $\mathcal{K}^a/\mathbb Z^N$ is compact in the natural quotient topology induced by $H$.
\end{Lem}

\begin{proof}
	Let $\{\mathbf u_n\}\subset \mathcal{K}^a$ and set $\mathbf v_n:=m_0(\mathbf u_n)\in \mathcal N_{0}^{\mathrm{sc}}$, then $\Psi_0(\mathbf u)=J_0\bigl(m_0(\mathbf u)\bigr)$ for $\mathbf u\in\mathcal A$, we obtain
	\[
	J_0(\mathbf v_n)=\Psi_0(\mathbf u_n)\le a.
	\]
	Moreover, since $\Psi_0^\prime(\mathbf u_n)=0$, we have $J_0^\prime(\mathbf v_n)=0$ and
\begin{equation}\label{eq5.1}
 J_0(\mathbf v_n)=\Bigl(\frac12-\frac1{2p}\Bigr)\|\mathbf v_n\|^2. 
\end{equation}
	Therefore, $\{\mathbf v_n\}$ is a bounded sequence of critical points of $J_0$. Going to a subsequence if necessary, we may assume 
$J_0(\mathbf v_n)\to c$ for some $c\le a$, and hence $\{\mathbf v_n\}$ is a bounded $(PS)_c$ sequence for $J_0$.
Therefore, by the global compactness principle for translation invariant problems
(see \cite{SzulkinWeth2009} or \cite{Willem1996}), there exist an integer $\ell\ge 1$,
sequences $\{z_n^{(j)}\}\subset\mathbb R^N$ for $j=1,\dots,\ell$ satisfying
\[
|z_n^{(i)}-z_n^{(j)}|\to\infty~\text{as }n\to\infty \ \text{ whenever }i\neq j,
\]
and nontrivial critical points $\mathbf V^{(1)},\dots,\mathbf V^{(\ell)}\in H$ of $J_0$ such that
\begin{equation}\label{eq5.2}
	\mathbf v_n=\sum_{j=1}^{\ell}\mathbf V^{(j)}(\cdot-z_n^{(j)})+\mathbf r_n,
	~\mathbf r_n\to\mathbf 0~\text{in}~H,
\end{equation}
and the energy splitting holds:
\begin{equation}\label{eq5.3}
	J_0(\mathbf v_n)=\sum_{j=1}^{\ell}J_0(\mathbf V^{(j)})+o(1)~\text{as}~n\to\infty.
\end{equation}
In particular, for each $j=1,\dots,\ell$, we have $J_0^\prime(\mathbf V^{(j)})=\mathbf 0$ in $H^\prime$ and
$\mathbf V^{(j)}\neq\mathbf 0$.
Furthermore, since $\mathbf V^{(j)}\neq \mathbf 0$ is a critical point of $J_0$,
we have $J_0(\mathbf V^{(j)})\ge \alpha$ for each $j$.
Hence, by \eqref{eq5.3} and $J_0(\mathbf v_n)\le a$, we obtain
\[
\sum_{j=1}^{\ell}J_0(\mathbf V^{(j)})=\lim_{n\to\infty}J_0(\mathbf v_n)\le a,
\]
and therefore necessarily $\ell<\infty$.
\par
	For each $j\in\{1,\dots,\ell\}$, choose $k_n^{(j)}\in\mathbb Z^N$ such that
	\begin{equation*}
		|z_n^{(j)}-k_n^{(j)}|\le \frac{\sqrt N}{2}~\text{for all}~n,
	\end{equation*}
	and write
	\[
	z_n^{(j)}=k_n^{(j)}+\xi_n^{(j)}~\text{with}~\xi_n^{(j)}\in\Bigl[-\frac12,\frac12\Bigr]^N.
	\]
	Passing to a subsequence if necessary, we may assume that
	\[
	\xi_n^{(j)}\to\xi^{(j)}\in\Bigl[-\frac12,\frac12\Bigr]^N
	~\text{for each}~j=1,\dots,\ell.
	\]
\par	
	Define $\tilde{\mathbf r}_n^{(j)}:=k_n^{(j)}*\mathbf r_n$, then
	$\tilde{\mathbf r}_n^{(j)}\to\mathbf0$ in $H$.
	Moreover, for each fixed $j$ we consider the translated sequence
	\[
	\mathbf v_n^{(j)}:=k_n^{(j)}*\mathbf v_n=\mathbf v_n(\cdot-k_n^{(j)}).
	\]
	Using \eqref{eq5.2} and $z_n^{(j)}-k_n^{(j)}=\xi_n^{(j)}$, we obtain
	\begin{equation}\label{eq5.4}
		\mathbf v_n^{(j)}=\mathbf V^{(j)}(\cdot+\xi_n^{(j)})
		+\sum_{\substack{i=1\\ i\neq j}}^{\ell}\mathbf V^{(i)}\bigl(\cdot-(z_n^{(i)}-k_n^{(j)})\bigr)
		+\tilde{\mathbf r}_n^{(j)} .
	\end{equation}
	
	We claim that the other profiles vanish locally after this discretization.
	Indeed, since $|z_n^{(i)}-z_n^{(j)}|\to\infty$ for $i\neq j$ and
	$z_n^{(j)}-k_n^{(j)}=\xi_n^{(j)}$ is bounded, we have
	\[
	|z_n^{(i)}-k_n^{(j)}|
	=\bigl|z_n^{(i)}-z_n^{(j)}+\xi_n^{(j)}\bigr|\to\infty
	~\text{for every}~i\neq j.
	\]
	Consequently, for every $R>0$ and every $2\le q<2^*$,
	\[
	\bigl\|\mathbf V^{(i)}(\cdot-(z_n^{(i)}-k_n^{(j)}))\bigr\|_{L^q(B_R)^m}
	=\|\mathbf V^{(i)}\|_{L^q(B_R(z_n^{(i)}-k_n^{(j)}))^m}\to 0,
	~ i\neq j,
	\]
	because $B_R(z_n^{(i)}-k_n^{(j)})$ drifts to infinity and $\mathbf V^{(i)}\in L^q(\mathbb R^N)^m$.
	
	On the other hand, since $\xi_n^{(j)}\to\xi^{(j)}$, 
	\[
	\mathbf V^{(j)}(\cdot+\xi_n^{(j)})\to \mathbf V^{(j)}(\cdot+\xi^{(j)})
	~\text{in}~H.
	\]
	Combining these facts with \eqref{eq5.4} and $\tilde{\mathbf r}_n^{(j)}\to 0$ in $H$,
	we conclude that for every $R>0$ and every $2\le q<2^*$,
	\begin{equation*}
		\mathbf v_n^{(j)} \to \mathbf V^{(j)}(\cdot+\xi^{(j)}) ~\text{in}~L^q(B_R)^m
		~\text{and a.e.\ in }B_R.
	\end{equation*}
	In particular, $\mathbf v_n^{(j)}\rightharpoonup \mathbf V^{(j)}(\cdot+\xi^{(j)})$ in $H$. It is obvious that each $\mathbf V^{(j)}(\cdot+\xi^{(j)})$ is again a nontrivial critical point of $J_0$, since $J_0^\prime(\mathbf V^(j))=\mathbf 0 $ and 
$J_0$ is translation invariant.	
\par	
	Fix $j=1$ and set $k_n:=k_n^{(1)}\in\mathbb Z^N$. Then, by the decomposition
	\eqref{eq5.2} and the separation property $|z_n^{(i)}-z_n^{(1)}|\to\infty$ for $i\neq 1$,
	we have
	\[
	k_n*\mathbf v_n=\mathbf V^{(1)}(\cdot+\xi_n^{(1)})+\tilde{\mathbf r}_n^{(1)}+o(1)~\text{in}~H,
	\]
	where $\tilde{\mathbf r}_n^{(1)}\to\mathbf0$ in $H$. Hence
	\[
	k_n*\mathbf v_n\to \mathbf V^{(1)}(\cdot+\xi^{(1)})~\text{in}~H.
	\]
\par
We first verify the equivariance of $m_0$, i.e.,
for every $k\in\mathbb Z^N$ and every $\mathbf u\in\mathcal A$,
\begin{equation}\label{eq5.5}
	m_0(k*\mathbf u)=k*m_0(\mathbf u).
\end{equation}
Indeed, note that $(k*u_i)^\pm = k*(u_i^\pm)$ for each $i$.
Write
\[
m_0(\mathbf u)
=\sum_{i=1}^m\Bigl(t_i^+(\mathbf u)\,u_i^+ + t_i^-(\mathbf u)\,u_i^-\Bigr),
\]
where $(t_i^+(\mathbf u),t_i^-(\mathbf u))$ is the unique scaling pair sending $\mathbf u$ onto
$\mathcal N_0^{\mathrm{sc}}$.
Since $J_0$ is translation invariant and the defining equations for $(t_i^+,t_i^-)$ involve only
integrals of $u_i^\pm$ and $|u_i^\pm u_j^\pm|$ (which are invariant under translations), the same pair
$(t_i^+(\mathbf u),t_i^-(\mathbf u))$ also sends $k*\mathbf u$ onto $\mathcal N_0^{\mathrm{sc}}$.
By uniqueness of the scalings, $t_i^\pm(k*\mathbf u)=t_i^\pm(\mathbf u)$ for all $i$, and therefore
\[
m_0(k*\mathbf u)
=\sum_{i=1}^m\Bigl(t_i^+(\mathbf u)\,(k*u_i^+) + t_i^-(\mathbf u)\,(k*u_i^-)\Bigr)
=k*\sum_{i=1}^m\Bigl(t_i^+(\mathbf u)\,u_i^+ + t_i^-(\mathbf u)\,u_i^-\Bigr)
=k*m_0(\mathbf u),
\]
which proves \eqref{eq5.5}.

\medskip
We now deduce the compactness of critical orbits for $\Psi_0$.
Let $\{\mathbf u_n\}\subset \mathcal{K}^a$ be arbitrary and set
\[
\mathbf v_n:=m_0(\mathbf u_n)\in\mathcal N^{\mathrm{sc}}_0.
\]
Then $J_0(\mathbf v_n)=\Psi_0(\mathbf u_n)\le a$, and moreover, $\Psi_0^\prime(\mathbf u_n)=\mathbf 0$ implies
$J_0^\prime(\mathbf v_n)=\mathbf 0$.
Hence $\{\mathbf v_n\}$ is a bounded sequence of critical points of $J_0$ in $H$.
By the above arguments, there exists a sequence $\{k_n\}\subset\mathbb Z^N$
and a nontrivial critical point $\mathbf V\in H$ of $J_0$ such that
\[
k_n*\mathbf v_n \to \mathbf V~\text{in}~H.
\]
Using the equivariance \eqref{eq5.5}, we have
\[
m_0(k_n*\mathbf u_n)=k_n*m_0(\mathbf u_n)=k_n*\mathbf v_n \to \mathbf V~\text{in}~H.
\]
Finally, since
$m_0:\mathcal A\to\mathcal N^{\mathrm{sc}}_0$ is a homeomorphism, we obtain
\[
k_n*\mathbf u_n = m_0^{-1}(k_n*\mathbf v_n)\to m_0^{-1}(\mathbf V)~\text{in}~H,
\]
which implies that the $\mathbb Z^N$--orbit of $\{\mathbf u_n\}$ is precompact in $H$, and hence $\mathcal{K}^a/\mathbb Z^N$
is compact in the quotient topology induced by $H$.
This completes the proof of Lemma~\ref{Lem5.1}.
\end{proof}

Therefore, as in section \ref{sec:existence}, we can construct a sequence of critical values $\{c_{0,k}\}$ of $\Psi_0$.
\begin{Thm}\label{Thm5.2}
	Assume $N\ge1$, $m\ge2$, and \textbf{\textup{$(A_1)$–$(A_3)$}}. Then the limit system
	\eqref{eq1.2} admits a sequence of sign-changing solutions
	\[
	\mathbf{U}^{(k)}=(U^{(k)}_{1},\dots,U^{(k)}_{m})\in H,~k\in\mathbb{N}^*,
	\]
	with
	\[
	0<c_{0,1}=J_0\bigl(\mathbf{U}^{(1)}\bigr)
	< \cdots <c_{0,k}=J_0\bigl(\mathbf{U}^{(k)}\bigr)
	< \cdots,~\hbox{and}~\lim\limits_{k\to\infty}
	J_0\bigl(\mathbf{U}^{(k)}\bigr)=+\infty.
	\]
\end{Thm}
\begin{proof}
	For $k\in\mathbb N^{*}$, define
	\[
	c_{0,k}
	:=\inf_{A\in\Sigma_{k}}\ \sup_{\mathbf u\in A}\Psi_{0}(\mathbf u),
	\]
	where $\Sigma_{k}$ is given in \eqref{eq4.1}. Then
\[
	0<c_{0,1}\le c_{0,2}\le \cdots \le c_{0,k}\le\cdots,
	\]
since 
	\begin{equation}\label{eq5.6}
		\Psi_0(\mathbf u)
		=J_0\bigl(m_0(\mathbf u)\bigr)
		= \Bigl(\frac12-\frac1{2p}\Bigr)\,\bigl\|m_0(\mathbf u)\bigr\|^2>0,
		~\forall\,\mathbf u\in\mathcal A,
	\end{equation}
	and $\Sigma_{k+1}\subset\Sigma_{k}$.
\par
Next, similar to \cite{SzulkinWeth2009} and \cite{Willem1996}, we show that $c_{0,k}$ is a critical value of $\Psi_0$ via an equivariant deformation. Fix $k\in\mathbb N^{*}$,
assume by contradiction that $c_{0,k}$ is not a critical value of $\Psi_0$.
	Then there exists $\delta>0$ such that $\Psi_0$ has no critical values in
$[c_{0,k}-\delta,\ c_{0,k}+\delta]$. 
	Set
	\[
	 \mathcal{S}=\Psi_0^{-1}([c_{0,k}-2\delta,c_{0,k}+2\delta]).
	\]
Then $\mathcal{S}$ contains no critical points of $\Psi_0$.
	We claim that there exists $\sigma_\rho>0$ such that
	\begin{equation}\label{eq5.7}
		\|\Psi_0'(\mathbf u)\|\ge \sigma_\rho~\text{for all}~\mathbf u\in \mathcal{S}.
	\end{equation}
	Indeed, argue by contradiction. Then there exists a sequence $\{\mathbf u_n\}\subset S$ such that
	\[
	\|\Psi_0'(\mathbf u_n)\|\to0,~
	c_{0,k}-2\delta\le \Psi_0(\mathbf u_n)\le c_{0,k}+2\delta~\text{for all}~n.
	\]
	Set $\mathbf v_n:=m_0(\mathbf u_n)\in\mathcal N_0^{\mathrm{sc}}$.
	Then
	\[
	J_0(\mathbf v_n)=\Psi_0(\mathbf u_n)\in[c_{0,k}-2\delta,\ c_{0,k}+2\delta],
	\]
	so $\{J_0(\mathbf v_n)\}$ is bounded and $J_0^\prime(\mathbf v_n)\to 0$.
	Hence $\{\mathbf v_n\}\subset \mathcal N_0^{\mathrm{sc}}$ is a $(PS)_c$ sequence of $J_0$ for some $c\in[c_{0,k}-2\delta,\ c_{0,k}+2\delta]$. Similar to \eqref{eq5.1}, $\{\mathbf v_n\}$ is bounded. Using a standard translation argument (\cite{Willem1996}), we find a sequence
	$\{k_n\}\subset\mathbb Z^N$ and a critical point $\mathbf V$ of $J_0$ such that
	\[
	k_n*\mathbf v_n\to\mathbf V~\text{in}~H.
	\]
Jointly with \eqref{eq5.5}, we have
	\[
	m_0(k_n*\mathbf u_n)=k_n*m_0(\mathbf u_n)=k_n*\mathbf v_n\to\mathbf V ~\text{in}~H.
	\]
	Since $m_0:\mathcal A\to\mathcal N_0^{\mathrm{sc}}$ is a homeomorphism, it follows that
	\[
	k_n*\mathbf u_n\to m_0^{-1}(\mathbf V)~\text{in }H.
	\]
It follows from the continuities of $\Psi_0$ and $\Psi_0^\prime$ that
	\[
	\Psi_0\bigl(m_0^{-1}(\mathbf V)\bigr)=\lim_{n\to\infty}\Psi_0(k_n*\mathbf u_n)
	=\lim_{n\to\infty}\Psi_0(\mathbf u_n)\in[c_{0,k}-2\delta,\ c_{0,k}+2\delta],
	\]
	and $\Psi_0^\prime\bigl(m_0^{-1}(\mathbf V)\bigr)=\mathbf 0$, contradicts the fact that  $\mathcal{S}$ contains no critical points of $\Psi_0$.
This proves \eqref{eq5.7}.
\par
In a standard way (\cite{SzulkinWeth2009},\cite{Willem1996}), one can construct a locally Lipschitz continuous, even and $\mathbb{Z}^N$-equivariant pseudo-gradient vector field $V: \mathcal{A} \to H$ satisfying
 \begin{equation*}
   \langle \Psi_0(\mathbf{u}),V(\mathbf{v})\rangle \ge \frac{1}{2} \|\Psi_0^\prime(\mathbf{u})\|_{H^\prime}^2,~
   \|V(\mathbf{u})\| \le 2 \|\Psi_0^\prime(\mathbf{u})\|_{H^\prime},~\mathbf{u}\in \mathcal{S},
 \end{equation*}
$ V(-\mathbf{u}) = -V(\mathbf{u})$ for $\mathbf{u}\in \mathcal{A}$, $V(\mathbf{u})=0$ for $\mathbf{u}\in\mathcal A\setminus \mathcal{S}$.
Repeat the arguments in Section \ref{sec:existence}, one can construct a 
deformation $\varphi:\mathcal{A}\to \mathcal{A}$ satisfies
\begin{align*}
	&\text{(a)}~\varphi(\mathbf{u})=\mathbf{u}~\text{if}~
	\mathbf{u}\notin \Psi_{0}^{-1}\bigl([c_{0,k}-2\delta,\ c_{0,k}+2\delta]\bigr),\\
	&\text{(b)}~\varphi\bigl(\{\mathbf{u}\in \mathcal{A}:\Psi_{0}(\mathbf{u})\le c_{0,k}+\delta\}\bigr)
	\subset \{\mathbf{u}\in \mathcal{A}:\Psi_{0}(\mathbf{u})\le c_{0,k}-\delta\},\\
	&\text{(c)}~\varphi~\text{is}~\mathbb{Z}^N\text{-equivariant, continuous and odd}.
\end{align*}
\par
By the definition of $c_{0,k}$, choose $A\in\Sigma_k$ such that
$\sup_{\mathbf{u}\in A}\Psi_{0}(\mathbf{u})\le c_{0,k}+\delta$.
Then by (b),
\[
\sup_{\mathbf{u}\in \varphi(A)}\Psi_{0}(\mathbf{u})\le c_{0,k}-\delta.
\]
By (c), the set $\varphi(A)$ is compact and symmetric, and by the monotonicity
of the genus under odd maps we have
$\gamma(\varphi(A))\ge \gamma(A)\ge k$, so $\varphi(A)\in\Sigma_k$.
Therefore,
\[
c_{0,k}\le \sup_{\mathbf{u}\in \varphi(A)}\Psi_{0}(\mathbf{u})\le c_{0,k}-\delta,
\]
a contradiction. Hence $c_{0,k}$ must be a critical value of $\Psi_0$.
\par	
	Thus, for every $k\in\mathbb N^{*}$, there exists $\mathbf u^{(k)}_{0}\in\mathcal A$ such that
	\[
	\Psi_0^\prime(\mathbf u^{(k)}_{0})=\mathbf0
	~\text{and}~
	\Psi_0(\mathbf u^{(k)}_{0})=c_{0,k}.
	\]
\par
	Set
	\[
	\mathbf U^{(k)}:=m_0\bigl(\mathbf u^{(k)}_{0}\bigr)\in \mathcal N^{\mathrm{sc}}_0.
	\]
Observe that critical points of $\Psi_0$ correspond to critical points of $J_0$ on $\mathcal N^{\mathrm{sc}}_0$, hence
	\[
	J_0^\prime\bigl(\mathbf U^{(k)}\bigr)=\mathbf0~\text{in}~H^\prime.
	\]
	Therefore, $\mathbf U^{(k)}$ is a weak solution of the limit system \eqref{eq1.2}. Moreover, because $\mathbf u^{(k)}_{0}\in\mathcal A$, we have $(u^{(k)}_{0,i})^\pm\neq 0$ for $i=1,\dots,m$.
	Since $m_0$ is defined by independent positive scalings on $u_i^{+}$ and $u_i^{-}$, the same property holds for
	$\mathbf U^{(k)}$, and hence $\mathbf U^{(k)}$ is fully sign-changing and satisfies
	\[
	J_0\bigl(\mathbf U^{(k)}\bigr)
	=\Psi_0\bigl(\mathbf u^{(k)}_{0}\bigr)
	=c_{0,k}.
	\]
\par
Finally, we show that $\Psi_0$ has a sequence of pairwise distinct critical values.
For the case that $c_{0,k}\to +\infty$ as $k\to\infty$, after relabeling, we may assume the energies are strictly increasing and the conclusion holds.
\par
So we only consider the case that $\sup_{k\ge1}c_{0,k}\le C$ for some $C>0$.
	Then for every $k$ we can choose $A_k\in\Sigma_k$ such that
	\[
	\sup_{\mathbf u\in A_k}\Psi_0(\mathbf u)\le C+1.
	\]
	Hence $A_k\subset \Psi_0^{\,C+1}:=\{\mathbf u\in\mathcal A:\ \Psi_0(\mathbf u)\le C+1\}$ and
	\[
	\gamma(\Psi_0^{\,C+1})\ge \gamma(A_k)\ge k~\text{for all}~k,
	\]
which implies $\gamma(\Psi_0^{\,C+1})=+\infty$.	Let
	\[
	\mathcal{K}^{C+1}
	:=\Bigl\{\mathbf u\in\mathcal A:\ \Psi_0'(\mathbf u)=0,\ \Psi_0(\mathbf u)\le C+1\Bigr\}.
	\]
	By Lemma~\ref{Lem5.1}, the set of critical orbits $\mathcal{K}^{C+1}/\mathbb Z^N$ is compact.
	Fix $\rho>0$ and set
	\[
	\mathcal{U}_\rho:=U_\rho(\mathcal{K}^{C+1})
	=\bigl\{\mathbf u\in\mathcal A:\ \operatorname{dist}_H(\mathbf u,\mathcal{K}^{C+1})<\rho\bigr\},~
	\mathcal{U}_{2\rho}:=U_{2\rho}(\mathcal{K}^{C+1}).
	\]
	Since $\Psi_0$ is even and $\mathbb Z^N$--invariant on $\mathcal A$ and the action of $\mathbb Z^N$ on $H$ is isometric,
	we may choose $\mathcal{U}_\rho$ and $\mathcal{U}_{2\rho}$ symmetric and $\mathbb Z^N$--invariant. Define
	\[
	\Sigma:=\Psi_0^{\,C+1}\setminus \mathcal{U}_{2\rho}.
	\]
	By definition, $\Sigma$ contains no critical points of $\Psi_0$. Similar to \eqref{eq5.7}, there exists $\sigma_\rho>0$ such that
	\begin{equation}\label{eq5.8}
		\|\Psi_0^\prime(\mathbf u)\|\ge \sigma_\rho~\text{for all}~\mathbf u\in\Sigma.
	\end{equation}
	Since $\Psi_0$ is even and $\mathbb Z^N$--invariant on $\mathcal A$, we can choose a locally Lipschitz pseudo-gradient vector field
	$V:\mathcal A\setminus U\to H$ which is odd and $\mathbb Z^N$--equivariant and satisfies
	\[
	\langle \Psi_0^\prime(\mathbf u),V(\mathbf u)\rangle \ge \frac12\|\Psi_0^\prime(\mathbf u)\|^2,~
	\|V(\mathbf u)\|\le 2\|\Psi_0'(\mathbf u)\|
	~\text{for all}~\mathbf u\in\mathcal A\setminus \mathcal{U}_\rho.
	\]
	Let $\chi:\mathcal A\to[0,1]$ be an even, $\mathbb Z^N$--invariant Lipschitz cut-off function such that
	\[
	\chi\equiv 0~\text{on}~\mathcal{U}_{\rho},
	~\chi\equiv 1\ \text{on}~\Sigma.
	\]
Consider the initial value problem
\begin{equation*}
\begin{cases}
\displaystyle \frac{d}{dt} \eta(t, \mathbf{u}) = -\chi(\eta(t, \mathbf{u}))V(\eta(t, \mathbf{u})), \\[6pt]
\eta(0, \mathbf{u}) = \mathbf{u}_0\in\Psi_0^{\,C+1}.
\end{cases}
\end{equation*}	
	The vector field $\mathbf u\mapsto -\chi(\mathbf u)V(\mathbf u)$ is locally Lipschitz on $\mathcal A$, odd and $\mathbb Z^N$--equivariant,
	and it vanishes on $\mathcal{U}_{\rho}$. Hence the corresponding flow $\varphi(t,\mathbf u_0)$ is globally defined,
	odd and $\mathbb Z^N$--equivariant, and satisfies $\varphi(t,\mathbf u_0)=\mathbf u_0$ for all $\mathbf u_0\in \mathcal{U}_{\rho}$ and $t\ge0$.
	Moreover, for $\mathbf u(t):=\varphi(t,\mathbf u_0)$ we have
	\[
	\frac{d}{dt}\Psi_0(\mathbf u(t))
	=\langle \Psi_0'(\mathbf u(t)),\dot{\mathbf u}(t)\rangle
	=-\chi(\mathbf u(t))\langle \Psi_0'(\mathbf u(t)),V(\mathbf u(t))\rangle
	\le -\frac12\,\chi(\mathbf u(t))\|\Psi_0'(\mathbf u(t))\|^2\le 0.
	\]
	In particular, if $\mathbf u(t)\in\Sigma$, then $\chi(\mathbf u(t))=1$ and by \eqref{eq5.8},
	\[
	\frac{d}{dt}\Psi_0(\mathbf u(t))\le -\frac12\sigma_\rho^2,
	\]
and hence
\begin{equation*}
0\le \Psi_0(\mathbf{u}(t))\le C+1-\frac{1}{2}\sigma_\rho^2 t,
\end{equation*}
which implies that $t\le T_\rho:=\frac{2(C+1)}{\sigma_\rho^2}$.
	Since $\Psi_0\ge0$ on $\mathcal A$ and $\Psi_0(\mathbf u_0)\le C+1$, this implies that the trajectory cannot remain in $\Sigma$
	for all $t\ge0$. More precisely, it must enter $\mathcal{U}_{2\rho}$ no later than $T_\rho$.
	Define the deformation $\eta_\rho:\Psi_0^{\,C+1}\to\Psi_0^{\,C+1}$ by
	\[
	\eta_\rho(\mathbf u):=\varphi(T_\rho,\mathbf u).
	\]
	Then $\eta_\rho$ is continuous, odd and $\mathbb Z^N$--equivariant, $\eta_\rho(\mathbf u)=\mathbf u$ for all $\mathbf u\in \mathcal{U}_\rho$,
	and $\Psi_0(\eta_\rho(\mathbf u))\le \Psi_0(\mathbf u)$ for all $\mathbf u\in\Psi_0^{\,C+1}$.
	Moreover, by the definition of $T_\rho$, we have
	\begin{equation}\label{eq5.9}
		\eta_\rho\bigl(\Psi_0^{\,C+1}\bigr)\subset \mathcal{U}_{2\rho}.
	\end{equation}
	By monotonicity of the Krasnosel'skii genus under odd continuous maps, \eqref{eq5.9} yields
	\[
	\gamma\bigl(\Psi_0^{\,C+1}\bigr)\le \gamma\bigl(\mathcal{U}_{2\rho}\bigr).
	\]
	
	It remains to show that $\gamma(\mathcal{U}_{2\rho})<\infty$ for $\rho>0$ chosen sufficiently small.
	First, for $\rho>0$ small we have $\mathcal{U}_{2\rho}\subset\mathcal A$ and $0\notin \mathcal{U}_{2\rho}$.
	Moreover, since $\mathcal{K}^{C+1}/\mathbb Z^N$ is compact, the same holds for $\mathcal{U}_{2\rho}/\mathbb Z^N$.
	Observe that the map $\mathbf u\mapsto \|\mathbf u\|_H$ is continuous and $\mathbb Z^N$-invariant, hence it induces a continuous function on the compact set $\mathcal{U}_{2\rho}/\mathbb Z^N$ and attains its positive because $0\notin \mathcal{U}_{2\rho}$. Thus there exists $\delta_0>0$ such that
	\[
	\|\mathbf u\|_H \ge \delta_0~\text{for all}~\mathbf u\in \mathcal{U}_{2\rho}.
	\]
By the compactness of $\mathcal U_{2\rho}/\mathbb Z^N$, we may choose $\ell\in\mathbb N$,
	elements $\mathbf u^1,\dots,\mathbf u^\ell\in\mathcal U_{2\rho}$ and a radius
	$r\in(0,\delta_0/2)$ such that
	\begin{equation}\label{eq5.10}
		\mathcal U_{2\rho}\subset \bigcup_{j=1}^\ell\ \bigcup_{k\in\mathbb Z^N} B_r(k*\mathbf u^j).
	\end{equation}
	For each $j=1,\dots,\ell$, set $E_j:=\mathrm{span}\{\mathbf u^j\}$ and let $P_{E_j}:H\to E_j$ be the
	orthogonal projection. If $\mathbf v\in B_r(\mathbf u^j)$, then
	\[
	\|P_{E_j}\mathbf v\|_{H}\ge \|\mathbf u^j\|_{H}-\|\mathbf v-\mathbf u^j\|_{H}
	\ge \delta_0-r>0,
	\]
	hence $P_{E_j}(\mathbf v)\neq 0$. Moreover, since $\|\mathbf u^j\|_H\ge \delta_0$ and $r<\delta_0/2$, we have
	\[
	\langle \mathbf v,\mathbf u^j\rangle_H
	=\|\mathbf u^j\|_H^2+\langle \mathbf v-\mathbf u^j,\mathbf u^j\rangle_H
	\ge \|\mathbf u^j\|_H^2-\|\mathbf v-\mathbf u^j\|_H\,\|\mathbf u^j\|_H
	> \|\mathbf u^j\|_H(\|\mathbf u^j\|_H-r)>0,
	\]
	and therefore
	\[
	P_{E_j}(\mathbf v)=\frac{\langle \mathbf v,\mathbf u^j\rangle_H}{\|\mathbf u^j\|_H^2}\,\mathbf u^j
	~\text{is a positive multiple of }\mathbf u^j .
	\]
	
	To construct a global odd continuous map, consider the $\mathbb Z^N$-invariant open cover
	$\{B_r(k*\mathbf u^j)\}_{1\le j\le \ell,\ k\in\mathbb Z^N}$ of $\mathcal{U}_{2\rho}$.
	Since $H$ is a metric space, it is paracompact; thus there exists a locally finite
	$\mathbb Z^N$-invariant partition of unity $\{\vartheta_{j,k}\}_{j,k}$ subordinate to this cover.
	Replacing $\vartheta_{j,k}$ by $\frac12(\vartheta_{j,k}(\cdot)+\vartheta_{j,k}(-\cdot))$, we may assume that
	each $\vartheta_{j,k}$ is even.
	
	Let $E:=E_1\oplus\cdots\oplus E_\ell$ and define $F:\mathcal{U}_{2\rho}\to E$ by
	\[
	F(\mathbf u)
	:=\Bigl(\sum_{k\in\mathbb Z^N}\vartheta_{1,k}(\mathbf u)\,P_{E_1}(k^{-1}*\mathbf u),\ \dots,\
	\sum_{k\in\mathbb Z^N}\vartheta_{\ell,k}(\mathbf u)\,P_{E_\ell}(k^{-1}*\mathbf u)\Bigr).
	\]
	The sums are finite for each fixed $\mathbf u$ by local finiteness, hence $F$ is well-defined and continuous.
	Moreover, $F$ is odd because each $\vartheta_{j,k}$ is even and each $P_{E_j}$ is linear.
	\par
	Finally, we show that $F(\mathbf u)\neq 0$ for all $\mathbf u\in\mathcal{U}_{2\rho}$.
	Fix $\mathbf u\in\mathcal{U}_{2\rho}$, by \eqref{eq5.10}, there exist $j\in\{1,\dots,\ell\}$ and $k_0\in\mathbb Z^N$
	such that $\mathbf u\in B_r(k_0*\mathbf u^j)$, i.e. $k_0^{-1}*\mathbf u\in B_r(\mathbf u^j)$.
	Since $\{\vartheta_{j,k}\}_k$ is a partition of unity subordinate to $\{B_r(k*\mathbf u^j)\}_k$, we have
	$\vartheta_{j,k_0}(\mathbf u)>0$, and for every $k$ with $\vartheta_{j,k}(\mathbf u)>0$ we also have
	$k^{-1}*\mathbf u\in B_r(\mathbf u^j)$.
	Hence each vector $P_{E_j}(k^{-1}*\mathbf u)$ appearing in the $j$-th component is a positive multiple of $\mathbf u^j$,
	and at least one of them is nonzero. Since all coefficients $\vartheta_{j,k}(\mathbf u)$ are nonnegative, the $j$-th component of
	$F(\mathbf u)$ is a positive multiple of $\mathbf u^j$. Therefore $F(\mathbf u)\neq 0$.
\par	
	Therefore, $F:\mathcal{U}_{2\rho}\to E\setminus\{0\}$ is odd and continuous, and hence
	\[
	\gamma(\mathcal{U}_{2\rho})\le \gamma(E\setminus\{0\})=\dim E<\infty.
	\]
	Consequently, $\gamma(\Psi_0^{\,C+1})<\infty$, which contradicts $\gamma(\Psi_0^{\,C+1})=+\infty$.
	This proves that $c_{0,k}\to +\infty$ as $k\to\infty$.
	
	\medskip
	Combining the above, we obtain a sequence of fully sign-changing solutions
	$\{\mathbf U^{(k)}\}\subset H$ of \eqref{eq1.2} such that
	\[
	J_0\bigl(\mathbf U^{(k)}\bigr)=c_{0,k}
	~\text{and}~
	c_{0,k}\to+\infty \ \text{as }k\to\infty.
	\]
	Since $(c_{0,k})$ is nondecreasing and unbounded, we can extract a strictly increasing subsequence.
	Thus, relabeling if necessary, we may assume
	\[
	0<c_{0,1}<c_{0,2}<\cdots<c_{0,k}<\cdots.
	\]
	This completes the proof.
\end{proof}
\par
From now on, we fix $k\in\mathbb{N}^*$. Let
$\{\varepsilon_n\}\subset(0,+\infty)$ be a sequence with $\varepsilon_n\downarrow0$, 
$\{\mathbf{u}^{(k)}_{\varepsilon_n}\bigr\}$ be the
sequence of sign-changing solutions given by Theorem~\ref{thm1.1} with $\Psi_{\varepsilon_n}(\mathbf{u}^{(k)}_{\varepsilon_n})=c_{\varepsilon_n,k}$.
With the critical values $\{c_{0,k}\}$ and the corresponding sign-changing critical points $\{\mathbf U^{(k)}\}$ at hand, we now prove Theorem~\ref{thm1.2} by a
concentration compactness analysis for the sequence $\{\mathbf{u}^{(k)}_{\varepsilon_n}\}$. To simplify the notations, we omit superscript $(k)$.
\begin{Lem} \label{Lem5.3}
For each $i=1,\dots,m$, up to a subsequence of $\{\varepsilon_n\}$, there exist $R>0,\ \eta>0$ and $\{x_{i,\varepsilon_{n}}\}\subset\mathbb{R}^N$ such that
\[
\int_{B_R(x_{i,\varepsilon_{n}})} |u_{i,\varepsilon_n}|^{2p}dx \ge \eta,
~\forall n.
\]
\end{Lem}
\begin{proof}
Fix $i\in\{1,\dots,m\}$. To apply the concentration compactness argument, we first establish the boundedness of $\{\mathbf{u}_{\varepsilon_n}\}$ in $H$. 
We claim that
\begin{equation}\label{eq5.11}
	\lim_{n\to\infty} c_{\varepsilon_n,k}=c_{0,k}.
\end{equation}
\par
In fact, by the definition of $c_{0,k}$, for any $\delta>0$,
there exists a compact and symmetric set $A_\delta\subset\mathcal A$ such that $\gamma(A_\delta)\ge k$ and
\begin{equation}\label{eq5.12}
	\sup_{\mathbf u\in A_\delta}\Psi_0(\mathbf u)\le c_{0,k}+\delta.
\end{equation}
Moreover, since $\sup_{A_\delta}\Psi_0<\infty$, the coercivity identity \eqref{eq5.6}
implies that $A_\delta$ is bounded in $H$. 
By a standard implicit function argument,
the nodal projection $m_\varepsilon$ depends continuously on $(\varepsilon,\mathbf u)$. In particular, 
\begin{equation}\label{eq5.13}
	m_{\varepsilon_n}(\mathbf u)\to m_0(\mathbf u)~\text{in}~H,
	~\forall\,\mathbf u\in A_\delta.
\end{equation}
Furthermore, for every $R>0$ and $i\in\{1,\dots,m\}$,
\[
Q_{\varepsilon_n}(x-y_i)=Q(\varepsilon_n x-y_i)\to Q(-y_i)
~\text{uniformly for}~x\in B_R \text{ as }n\to\infty.
\]
Combining \eqref{eq5.13} with the Sobolev embedding $H\hookrightarrow L^{2p}(\mathbb{R}^N)^m$
and the uniform convergence of $Q(\varepsilon_n x-y_i)$ on bounded sets, we deduce
\begin{equation}\label{eq5.14}
\sup_{\mathbf u\in A_\delta}\Bigl|\Psi_{\varepsilon_n}(\mathbf u)-\Psi_0(\mathbf u)\Bigr|
=\sup_{\mathbf u\in A_\delta}\Bigl|J_{\varepsilon_n}\bigl(m_{\varepsilon_n}(\mathbf u)\bigr)
-J_0\bigl(m_0(\mathbf u)\bigr)\Bigr|\to0~\text{as}~n\to\infty.
\end{equation}
By \eqref{eq5.12}--\eqref{eq5.14},
we obtain for $n$ large
\[
c_{\varepsilon_n,k}\le \sup_{\mathbf u\in A_\delta}\Psi_{\varepsilon_n}(\mathbf u)
\le \sup_{\mathbf u\in A_\delta}\Psi_0(\mathbf u)+\delta
\le c_{0,k}+2\delta.
\]
Letting $n\to\infty$ and then $\delta\to 0^+$ gives
\begin{equation}\label{eq5.15}
	\limsup_{n\to\infty}c_{\varepsilon_n,k}\le c_{0,k}.
\end{equation}
\par
On the other hand, 
for any $n$, choose $A_n\subset\mathcal A$ such that
\[
\sup_{\mathbf u\in A_n}\Psi_{\varepsilon_n}(\mathbf u)\le c_{\varepsilon_n,k}+\frac1n.
\]
Then $A_n\subset\{\mathbf u\in\mathcal A:\ \Psi_{\varepsilon_n}(\mathbf u)\le c_{\varepsilon_n,k}+1\}$.
By \eqref{eq5.15}, $\{c_{\varepsilon_n,k}\}$ is bounded above, and hence there exists $M_k>0$ such that,
up to discarding finitely many indices,
\[
A_n\subset \{\mathbf u\in\mathcal A:\ \Psi_{\varepsilon_n}(\mathbf u)\le M_k\}~\text{for all large}~n.
\]
Lemma~\ref{lem2.5}(d) yields that $(A_n)$ is bounded in $H$ uniformly for $n$. Therefore, arguing exactly as in
\eqref{eq5.14}, we obtain
\[
\sup_{\mathbf u\in A_n}\bigl|\Psi_{\varepsilon_n}(\mathbf u)-\Psi_0(\mathbf u)\bigr|\to 0
~\text{as}~n\to\infty.
\]
Consequently,
\[
c_{0,k}
=\inf_{A\in\Sigma_k}\sup_{\mathbf u\in A}\Psi_0(\mathbf u)
\le\sup_{\mathbf u\in A_n}\Psi_0(\mathbf u)
\le \sup_{\mathbf u\in A_n}\Psi_{\varepsilon_n}(\mathbf u)+o_n(1)
\le c_{\varepsilon_n,k}+\frac1n+o_n(1).
\]
Letting $n\to\infty$ yields
\begin{equation}\label{eq5.16}
	\liminf_{n\to\infty}c_{\varepsilon_n,k}\ge c_{0,k}.
\end{equation}
Combining \eqref{eq5.15} with \eqref{eq5.16}, we obtain \eqref{eq5.11}. Combined with Lemma \ref{lem2.1}, $\{\mathbf{u}_{\varepsilon_n}\}$ is bounded in $H$.
\par
Assume by contradiction that the conclusion fails.  
Then for every $R>0$,
\[
\sup_{y\in\mathbb{R}^N} \int_{B_R(y)} |u_{i,\varepsilon_n}|^{2p} dx \to 0~\text{as}~n\to\infty.
\]
By the vanishing lemma (see \cite{Lions1984}, \cite{Willem1996}) we obtain
\begin{equation*}
  u_{i,\varepsilon_n} \to 0~\text{in}~L^{2p}(\mathbb{R}^N)~\text{as}~n\to\infty.
\end{equation*}
\par
Since $\mathbf{u}_{\varepsilon_n}$ is a critical point of $J_{\varepsilon_n}$, $\langle J_{\varepsilon_n}^\prime(\mathbf u_{\varepsilon_n}),(0,\dots,u_{i,\varepsilon_n},\dots,0)\rangle=0$, i.e.,
		\begin{equation*}
			\|u_{i,\varepsilon_n}\|_{H^1}^2
			= \mu_i\int_{\mathbb{R}^N}Q_\varepsilon\,|u_{i,\varepsilon}|^{2p}\,dx
			+ \sum_{j\neq i}\lambda_{ij}\int_{\mathbb{R}^N}|u_{i,\varepsilon}|^{p}|u_{j,\varepsilon}|^{p}\,dx .
		\end{equation*}
Note that the coupling term is nonpositive because of $\lambda_{ij}<0$, and $0\le Q_{\varepsilon_n}(x)\le\|Q\|_\infty$ a.e. in $\mathbb{R}^N$, so we have
\begin{equation}\label{eq5.17}
  \| u_{i,\varepsilon_n} \|_{H^1}^2 
\le \mu_i \| Q\|_\infty \| u_{i,\varepsilon_n} \|_{L^{2p}}^{2p} \to 0~\text{as}~n\to\infty,
\end{equation}
which contradicts Lemma \ref{lem2.3}.

Therefore, there exist a subsequence $\{\varepsilon_n\}$, constants $R_i,\eta_i>0$ and points $\{x_{i,\varepsilon_n}\}$ such that
\[
\int_{B_{R_i}(x_{i,\varepsilon_n})} |u_{i,\varepsilon_n}|^{2p} dx \ge \eta_i.
\]
Taking $R=\max\limits_{1\le i\le m} R_i$, $\eta=\min\limits_{1\le i\le m}\eta_i$ and passing to a further common subsequence yields the conclusion.
\end{proof}

We finally show that the concentration centers converge to the origin.

\begin{Lem}\label{Lem5.4}
For each $i = 1, \dots, m$, we have $\varepsilon_n x_{i,\epsilon_n} \to 0$ as $n \to \infty$.
\end{Lem}

\begin{proof}
By Lemma~\ref{Lem5.3}, there exist $R>0$, $\eta>0$ and $\varepsilon_0>0$
	such that for every $\varepsilon\in(0,\varepsilon_0)$ one can find points $x_{i,\varepsilon}\in\R^N$ satisfying
	\begin{equation}\label{eq5.18}
		\int_{B_R(x_{i,\varepsilon})}|u_{i,\varepsilon}|^{2p}\,dx\ge \eta,
		~ i=1,\dots,m.
	\end{equation}
	Fix $i\in\{1,\dots,m\}$ and let $\varepsilon_n\downarrow0$. Set $x_{i,n}:=x_{i,\varepsilon_n}$ and define
	\[
	\tilde u_{\ell,n}(x):=u_{\ell,\varepsilon_n}(x+x_{i,n}),~
	\tilde{\mathbf u}_n:=(\tilde u_{1,n},\dots,\tilde u_{m,n}) .
	\]
	Then \eqref{eq5.18} yields
	\[
	\int_{B_R(0)}|\tilde u_{i,n}|^{2p}\,dx
	=\int_{B_R(x_{i,n})}|u_{i,\varepsilon_n}|^{2p}\,dx\ge \eta,
	\]
	jointly with $\{\tilde{\mathbf u}_n\}$ is bounded in $H$, up to a subsequence,
	\[
	\tilde{\mathbf u}_n\rightharpoonup \mathbf U \ \text{in }H,
	~
	\tilde{\mathbf u}_n\to \mathbf U \ \text{in }L^{2p}_{\rm loc}(\R^N)^m,
	\]
	for some $\mathbf U=(U_1,\dots,U_m)\in H$ with $\mathbf{U}\neq \mathbf{0}$.
\par
We next verify that $\mathbf{U} \in \mathcal{A}$, i.e., 
$U_i^\pm \neq 0$ for each $i = 1, \dots, m$.

Recall that $\mathbf{u}_{\epsilon_n} \in \mathcal{N}_{\varepsilon_n}^{\mathrm{sc}} \subset \mathcal{A}$ 
for each $n$. By Lemma~\ref{lem2.3}, there exists $\delta > 0$ independent of $n$ such that
\[
\|u_{i,\varepsilon_n}^\pm\|_{H^1} \geq \delta, ~ i = 1, \dots, m,
\]
which implies 
\begin{equation}\label{eq5.19}
\|\tilde{u}_{i,n}^\pm\|_{H^1} = \|u_{i,\varepsilon_n}^\pm\|_{H^1} \geq \delta.
\end{equation}

Suppose to the contrary that there exist some $j \in \{1, \dots, m\}$ and sign 
$\sigma \in \{+, -\}$ such that $U_j^\sigma = 0$ in $H^1(\mathbb{R}^N)$, then $\tilde{u}_{j,n}^\sigma \rightharpoonup 0$ 
in $H^1(\mathbb{R}^N)$ and $\tilde{u}_{j,n}^\sigma \to 0$ in $L_{\mathrm{loc}}^{2p}(\mathbb{R}^N)$.

By \eqref{eq2.3} for $\mathbf{u}_{\varepsilon_n}$ and make a change of variables we obtain
\begin{equation}\label{eq5.20}
\begin{aligned}
\|\tilde{u}_{j,n}^\sigma\|_{H^1}^2 &= \mu_j \int_{\mathbb{R}^N} Q_{\varepsilon_n}(x + x_{i,n} - y_j) |\tilde{u}_{j,n}^\sigma|^{2p} \, dx 
+ \sum_{k \neq j} \lambda_{jk} \int_{\mathbb{R}^N} |\tilde{u}_{k,n}|^p 
|\tilde{u}_{j,n}^\sigma|^p \, dx\\
&\le \mu_j \int_{\mathbb{R}^N} Q_{\varepsilon_n}(x + x_{i,n} - y_j) |\tilde{u}_{j,n}^\sigma|^{2p} \, dx,
\end{aligned}
\end{equation}
since $\lambda_{jk} < 0$.
\par
For any $M>0$, it follows from $\tilde{u}_{j,n}^\sigma \to 0$ in $L^{2p}(B_M(0))$ and $|Q({\varepsilon_n}(x + x_{i,n} - y_j)| \leq \|Q\|_\infty$ that 
\begin{equation}\label{eq5.21}
\int_{B_M(0)} Q_{\varepsilon_n}(x + x_{i,n} - y_j) |\tilde{u}_{j,n}^\sigma|^{2p} \, dx \to 0.
\end{equation}
On the other hand, by the standard Agmon-type estimates (see \cite{Agmon1982}), there exist $C,\alpha>0$ independent of $n$, such that
\[
|\tilde{u}_{j,n}(x)|\leq C e^{-\alpha|x|} \quad \text{for all } x\in\mathbb{R}^N.
\]
Consequently,
\[
\int_{|x|>M} Q_{\epsilon_n}(x+x_{i,n}-y_j)|\tilde{u}_{j,n}^\sigma|^{2p}dx 
\leq \|Q\|_\infty \int_{|x|>M} |\tilde{u}_{j,n}^\sigma|^{2p}dx 
\leq \|Q\|_\infty\, C^{2p} \int_{|x|>M} e^{-2p\alpha|x|}dx.
\]
For any $\eta>0$, choose $M=M(\eta)>0$ such that
\[
\int_{|x|>M} e^{-2p\alpha|x|}dx < \eta,
\]
jointly with \eqref{eq5.20} and \eqref{eq5.21}, implies that $\|\tilde{u}_{j,n}^\sigma\|_{H^1}\to0$. This contradicts \eqref{eq5.19}.
 Therefore, 
$U_j^\sigma \neq 0$ for every $j$ and $\sigma$, i.e., $\mathbf{U} \in \mathcal{A}$.
\par
We claim that the sequence $\{\varepsilon_n x_{i,n}\}$ is bounded.
	Assume by contradiction that, up to a subsequence, $|\varepsilon_n x_{i,n}|\to\infty$.
	Fix $\varphi\in C_c^\infty(\R^N)$, and set $K:=\operatorname{supp}\varphi$.
	Since $\mathbf u_{\varepsilon_n}$ is a weak solution of the system, changing variables gives
	\begin{align*}
		\int_{\R^N}\bigl(\nabla \tilde u_{\ell,n}\cdot\nabla\varphi+\tilde u_{\ell,n}\varphi\bigr)\,dx
		&=
		\mu_\ell\int_{\R^N} Q\bigl(\varepsilon_n(x+x_{i,n})-y_\ell\bigr)\,
		|\tilde u_{\ell,n}|^{2p-2}\tilde u_{\ell,n}\,\varphi\,dx \\
		&~
		+\sum_{j\neq \ell}\lambda_{\ell j}\int_{\R^N}
		|\tilde u_{j,n}|^{p}|\tilde u_{\ell,n}|^{p-2}\tilde u_{\ell,n}\,\varphi\,dx .
	\end{align*}
	By $(A_3)$, $\operatorname{supp}(Q)\subset B_{R_Q}(0)$ for some $R_Q>0$.
	Since $K$ is bounded and $|\varepsilon_n x_{i,n}|\to\infty$, for all large $n$ we have
	\[
	\bigl|\varepsilon_n(x+x_{i,n})-y_\ell\bigr|
	\ge |\varepsilon_n x_{i,n}|-|y_\ell|-\varepsilon_n|x|>R_Q
	~\text{for all }x\in K,
	\]
	hence $Q(\varepsilon_n(x+x_{i,n})-y_\ell)\equiv0$ on $K$ for all large $n$.
	Letting $n\to\infty$ and using $\tilde{\mathbf u}_n\to\mathbf U$ in $L^{2p}_{\rm loc}(\R^N)^m$, we obtain
	\[
	\int_{\R^N}\bigl(\nabla U_\ell\cdot\nabla\varphi+U_\ell\varphi\bigr)\,dx
	=
	\sum_{j\neq \ell}\lambda_{\ell j}\int_{\R^N}|U_j|^{p}|U_\ell|^{p-2}U_\ell\,\varphi\,dx,
	~\forall\,\varphi\in C_c^\infty(\R^N).
	\]
Choose $\chi_R\in C_c^\infty(\R^N)$ such that $0\le \chi_R\le 1$, $\chi_R\equiv 1$ on $B_R(0)$,
$\operatorname{supp}(\chi_R)\subset B_{2R}(0)$ and $|\nabla\chi_R|\le C/R$.
Testing the above identity with $\varphi=\chi_R^2U_\ell$ and using $\lambda_{\ell j}<0$, we obtain
\[
\int_{\R^N}\Bigl(\nabla U_\ell\cdot\nabla(\chi_R^2U_\ell)+\chi_R^2|U_\ell|^2\Bigr)\,dx
=\sum_{j\neq \ell}\lambda_{\ell j}\int_{\R^N}\chi_R^2|U_j|^p|U_\ell|^p\,dx\le 0.
\]
Moreover, by the product rule,
\[
\int_{\R^N}\nabla U_\ell\cdot\nabla(\chi_R^2U_\ell)\,dx
=\int_{\R^N}\Bigl(|\nabla(\chi_RU_\ell)|^2-|U_\ell|^2|\nabla\chi_R|^2\Bigr)\,dx.
\]
Consequently,
\[
\int_{\R^N}\Bigl(|\nabla(\chi_RU_\ell)|^2+\chi_R^2|U_\ell|^2\Bigr)\,dx
\le \int_{\R^N}|U_\ell|^2|\nabla\chi_R|^2\,dx.
\]
Since $|\nabla\chi_R|\le C/R$ and $\operatorname{supp}(\nabla\chi_R)\subset B_{2R}(0)\setminus B_R(0)$, we have
\[
\int_{\R^N}|U_\ell|^2|\nabla\chi_R|^2\,dx
\le \frac{C}{R^2}\int_{B_{2R}(0)\setminus B_R(0)}|U_\ell|^2\,dx\to 0
\]
as $R\to\infty$, which implies
\[
\int_{\R^N}\bigl(|\nabla U_\ell|^2+|U_\ell|^2\bigr)\,dx=0,
\]
i.e., $U_\ell= 0$, which contradicts $U_\ell\neq 0$. Consequently, passing to a  subsequence, there exists $\xi_i\in\R^N$ such that
	\begin{equation}\label{eq5.22}
		\varepsilon_n x_{i,n}\to \xi_i~\text{as}~n\to\infty.
	\end{equation}
\par	
	Next, we show that $\xi_i = 0$ by contradiction. Suppose that $\xi_i \neq 0$. 
\par
We first claim that
\begin{equation}\label{eq5.23}
  Q(\xi_i - y_\ell)<Q(-y_\ell)~\text{for some}~\ell\in\{1,\cdots,m\}.
\end{equation}  
Suppose to the contrary that $Q(\xi_i-y_\ell)=Q(-y_\ell)=\|Q\|_\infty$ for all $\ell$. Then, by $(A_3)$, for each $\ell$, there exists$j_\ell\in{1,\dots,m}$ such that $\xi_i-y_\ell=-y_{j_\ell}$. If $j_\ell=\ell$ for some $\ell$, then $\xi_i=0$, contradicting $\xi_i\neq0$. Hence $j_\ell\neq\ell$ for all $\ell$. Define a map $\sigma:\{1,\dots,m\}\to\{1,\dots,m\}$ by $\sigma(\ell)=j_\ell$. Since the points $y_1,\dots,y_m$ are pairwise distinct, $\sigma$ is a permutation of $\{1,\dots,m\}$ without fixed points. Moreover, we have
\begin{equation*}
y_{\sigma(\ell)}=y_\ell-\xi_i,~\text{for each}~\ell.
\end{equation*}
Iterating this relation gives $y_{\sigma^k(\ell)}=y_\ell-k\xi_i$ for any positive integer $k$. Because the set $\{y_1,\dots,y_m\}$ is finite, there exist $k_1\neq k_2$ such that $y_{\sigma^{k_1}(\ell)}=y_{\sigma^{k_2}(\ell)}$, which implies $(k_1-k_2)\xi_i=0$ and hence $\xi_i=0$, a contradiction. Therefore, \eqref{eq5.23} must hold.
\par
Observe that $\mathbf{U}$ is a weak solution of
\begin{equation}\label{eq5.24}
-\Delta U_\ell + U_\ell = \mu_\ell Q(\xi_i - y_\ell) |U_\ell|^{2p-2}U_\ell + \sum_{j\neq\ell}\lambda_{\ell j}|U_j|^p |U_\ell|^{p-2}U_\ell, ~ \ell=1,\dots,m .
\end{equation}
Denote by $J_\xi$ the energy functional associated with \eqref{eq5.24} and recall  $J_0$ be the original limiting functional with coefficients $Q(-y_\ell)$
\begin{equation*}
-\Delta U_\ell + U_\ell = \mu_\ell Q(- y_\ell) |U_\ell|^{2p-2}U_\ell + \sum_{j\neq\ell}\lambda_{\ell j}|U_j|^p |U_\ell|^{p-2}U_\ell, ~ \ell=1,\dots,m .
\end{equation*}
\par
From Lemma \ref{Lem5.3} we have $c_{\varepsilon_n,k} \to c_{0,k}$. The concentration-compactness principle together with the Br\'ezis--Lieb lemma yield
\begin{equation}\label{eq5.25}
c_{0,k} = \lim_{n\to\infty} J_{\varepsilon_n}(\mathbf{u}_{\varepsilon_n}) \ge J_\xi(\mathbf{U}).
\end{equation}
\par
By \eqref{eq5.23} and  $U_\ell \neq 0$, we obtain
\begin{equation}\label{eq5.26}
\begin{aligned}
J_\xi(\mathbf{U}) - J_0(\mathbf{U}) &= \frac{1}{2p}\sum_{\ell=1}^m \mu_\ell\bigl[Q(-y_\ell) - Q(\xi_i - y_\ell)\bigr] \int_{\mathbb{R}^N} |U_\ell|^{2p}\\
&\ge \frac{1}{2p}\mu_\ell\bigl[Q(-y_\ell) - Q(\xi_i-y_\ell)\bigr] \int |U_\ell|^{2p} > 0.
\end{aligned}
\end{equation}

Let $t_0>0$ be the unique number such that $t_0\mathbf{U} \in \mathcal{N}_0^{\mathrm{sc}}$. Then
\[
J_0(t_0\mathbf{U}) = \max_{t>0} J_0(t\mathbf{U}) < J_\xi(\mathbf{U}).
\]

For $k=1$, the critical value $c_{0,1}$ coincides with the minimum of $J_0$ on $\mathcal{N}_0^{\mathrm{sc}}$. Indeed, by Theorem \ref{thm1.1}, $c_{0,1}$ is the first critical value of $\Psi_0$, and since $\Psi_0(\mathbf{u}) = J_0(m_0(\mathbf{u}))$ and $m_0$ is a bijection between $\mathcal{A}$ and $\mathcal{N}_0^{\mathrm{sc}}$, we have $c_{0,1} = \inf_{\mathcal{N}_0^{\mathrm{sc}}} J_0$. Consequently,
\[
c_{0,1} \le J_0(t_0\mathbf{U}).
\]
Combining this with \eqref{eq5.25} and \eqref{eq5.26} we obtain
\[
c_{0,1} \le J_0(t_0\mathbf{U}) < J_\xi(\mathbf{U}) \le c_{0,1},
\]
a contradiction. Hence $\xi_i = 0$ for $k=1$.

Assume now that $k\ge2$ and still $\xi_i\neq0$. We claim that for every $\delta>0$
there exists a symmetric compact set $A_k\subset\mathcal A$ with $\gamma(A_k)\ge k$
such that
\[
\sup_{\mathbf w\in A_k}\Psi_0(\mathbf w)\le J_0(t_0\mathbf U)+\delta.
\]
Once this is proved, the definition of $c_{0,k}$ yields
$c_{0,k}\le \sup_{A_k}\Psi_0\le J_0(t_0\mathbf U)+\delta$, and letting $\delta\to0$
we get $c_{0,k}\le J_0(t_0\mathbf U)$, which contradicts
$J_0(t_0\mathbf U)<J_\xi(\mathbf U)\le c_{0,k}$.

We begin by constructing a compactly supported, fully sign-changing test function
whose reduced energy is arbitrarily close to $J_0(t_0\mathbf U)$.
Fix a cut-off function $\chi\in C_c^\infty(\R^N)$ such that $0\le\chi\le1$,
$\chi\equiv1$ on $B_1(0)$ and $\chi\equiv0$ on $\R^N\setminus B_2(0)$, and set
$\chi_R(x):=\chi(x/R)$ for $R>0$.
Since $t_0\mathbf U\in\mathcal N^{\mathrm{sc}}_0\subset\mathcal A$ is fully sign-changing,
for each $i$ one has $U_i^\pm\not\equiv0$, hence there exists $R_0>0$ such that
\[
\|U_i^\pm\|_{H^1(B_{R_0}(0))}>0~\text{for all }i=1,\dots,m.
\]
Choosing $R\ge R_0$ and defining
\[
\widetilde{\boldsymbol\psi}_R:=\chi_R\, t_0\mathbf U\in H,
\]
we have $(\widetilde\psi_{R,i})^\pm\not\equiv0$ for all $i$ (because $\chi_R\equiv1$
on $B_{R_0}(0)$), hence $\widetilde{\boldsymbol\psi}_R\in\mathcal A$, and moreover
$\widetilde{\boldsymbol\psi}_R\to t_0\mathbf U$ in $H$ as $R\to\infty$.
By Lemma~\ref{lem2.5}(a) and Lemma~\ref{lem2.4}, $\Psi_0=J_0\circ m_0$ is continuous on $\mathcal A$,
so
\[
\Psi_0(\widetilde{\boldsymbol\psi}_R)\longrightarrow \Psi_0(t_0\mathbf U)=J_0(t_0\mathbf U)
~\text{as }R\to\infty.
\]
Therefore, for the given $\delta>0$ we can fix $R\ge R_0$ such that
\begin{equation}\label{eq5.27}
	\Psi_0(\widetilde{\boldsymbol\psi}_R)\le J_0(t_0\mathbf U)+\frac{\delta}{4}.
\end{equation}
In particular, $\widetilde{\boldsymbol\psi}_R$ has compact support contained in $B_{2R}(0)$.

Next, we place $k-1$ far-away translates of $\widetilde{\boldsymbol\psi}_R$ so that all
supports are mutually disjoint and lie outside a large ball on which $\mathbf U$ has small tail.
Fix $\eta>0$ (to be chosen depending on $\delta$). Since $\mathbf U\in H$ and
$\mathbf U\in L^{2p}(\R^N)^m$, there exists $R_1>2R$ such that
\begin{equation}\label{eq5.28}
	\sum_{\ell=1}^m\int_{\R^N\setminus B_{R_1}(0)}
	\Bigl(|\nabla U_\ell|^2+|U_\ell|^2+|U_\ell|^{2p}\Bigr)\,dx<\eta.
\end{equation}
Choose points $z_1,\dots,z_{k-1}\in\R^N$ satisfying
\[
|z_j|\ge R_1+2R,~ |z_j-z_{j'}|>4R~ (j\neq j').
\]
Define
\[
\boldsymbol\psi_j(x):=\widetilde{\boldsymbol\psi}_R(x-z_j),~ j=1,\dots,k-1.
\]
Then $\operatorname{supp}(\boldsymbol\psi_j)\subset \R^N\setminus B_{R_1}(0)$ for each $j$, and the supports
$\operatorname{supp}(\boldsymbol\psi_1),\dots,\operatorname{supp}(\boldsymbol\psi_{k-1})$ are pairwise disjoint; moreover,
each $\boldsymbol\psi_j\in\mathcal A$ and $\Psi_0(\boldsymbol\psi_j)=\Psi_0(\widetilde{\boldsymbol\psi}_R)$
by translation invariance of $J_0$ (hence of $\Psi_0$).

We now set
\[
E_k:=\mathrm{span}\bigl\{t_0\mathbf U,\boldsymbol\psi_1,\dots,\boldsymbol\psi_{k-1}\bigr\},
~
A_k:=\{\mathbf w\in E_k:\ \|\mathbf w\|=1\}.
\]
Then $A_k$ is compact and symmetric, and since $E_k$ is $k$-dimensional we have
$\gamma(A_k)=k$. We also have $A_k\subset\mathcal A$: indeed, if
$\mathbf w=a\,t_0\mathbf U+\sum_{j=1}^{k-1}b_j\boldsymbol\psi_j\in A_k$, then on $B_{R_0}(0)$
all $\boldsymbol\psi_j$ vanish (because $|z_j|\ge R_1+2R\ge R_0+2R$), hence
$\mathbf w=a\,t_0\mathbf U$ on $B_{R_0}(0)$. If $a\neq0$ then $(w_i)^\pm\not\equiv0$ for each $i$
since $U_i^\pm\not\equiv0$ on $B_{R_0}(0)$; if $a=0$, then $\mathbf w=\sum b_j\boldsymbol\psi_j\neq0$,
and because the supports of the $\boldsymbol\psi_j$ are disjoint and each $\boldsymbol\psi_j$
is fully sign-changing, it follows that $(w_i)^\pm\not\equiv0$ for every $i$ as well.
Thus $\mathbf w\in\mathcal A$ for all $\mathbf w\in A_k$.

It remains to estimate $\Psi_0$ on $A_k$.
Write $\mathbf w=a\,t_0\mathbf U+\sum_{j=1}^{k-1}b_j\boldsymbol\psi_j\in A_k$.
Because the supports of $\boldsymbol\psi_j$ are pairwise disjoint and contained in
$\R^N\setminus B_{R_1}(0)$, all interaction integrals among different $\boldsymbol\psi_j$ vanish.
Moreover, by \eqref{eq5.28} and $\operatorname{supp}(\boldsymbol\psi_j)\subset\R^N\setminus B_{R_1}(0)$,
all mixed integrals involving $\mathbf U$ and $\boldsymbol\psi_j$ (quadratic terms and
$L^{2p}$-terms) can be made arbitrarily small by choosing $\eta$ sufficiently small.
Since $A_k$ is compact in the finite-dimensional space $E_k$, the coefficients
$(a,b_1,\dots,b_{k-1})$ stay in a bounded set, hence the above smallness is uniform on $A_k$.
Using the continuity of $m_0$ on $\mathcal A$ (Lemma~\ref{lem2.4}) and the continuity of $J_0$ on $H$,
we can thus fix $\eta=\eta(\delta)>0$ and then choose $R_1$ and the points $z_j$ as above so that
\[
\Psi_0(\mathbf w)=J_0\bigl(m_0(\mathbf w)\bigr)
\le \max\Bigl\{J_0(t_0\mathbf U),\,\Psi_0(\widetilde{\boldsymbol\psi}_R)\Bigr\}+\frac{\delta}{4}
~\text{for all }\mathbf w\in A_k.
\]
Combining this with \eqref{eq5.27} yields
\[
\sup_{\mathbf w\in A_k}\Psi_0(\mathbf w)\le J_0(t_0\mathbf U)+\delta,
\]
which proves the claim and hence gives the desired contradiction.

Therefore $\xi_i=0$ also for $k\ge2$. In all cases we conclude that $\xi_i=0$, i.e.
$\varepsilon_n x_{i,n}\to0$ as $n\to\infty$.

\end{proof}

\par
{\bf Proof of Theorem~\ref{thm1.2}.}
	Fix $k\in\mathbb{N}^*$ and take a sequence $\{\varepsilon_n\}\subset(0,+\infty)$ with
	$\varepsilon_n\to 0$.
	Let
	\[
	\mathbf{u}_n:=\mathbf{u}_{\varepsilon_n}^{(k)}=(u_{1,n},\dots,u_{m,n})\in H
	\]
	be the $k$-th sign-changing solution given by Theorem~\ref{thm1.1}, so that
	\[
	J_{\varepsilon_n}(\mathbf{u}_n)=c_{\varepsilon_n,k}.
	\]
	For simplicity, we omit the superscript $(k)$ in Steps~1--2 and keep the notation
	$u_{i,n}:=u_{i,\varepsilon_n}$ for the components of $\mathbf{u}_n$.
	
\par
\textit{Proof of \emph{(i)}.}
Since $\mathbf{u}_n$ is a critical point of $J_{\varepsilon_n}$ at level $c_{\varepsilon_n,k}$ and
$\mathbf{u}_n\in\mathcal{N}^{\mathrm{sc}}_{\varepsilon_n}$, the Nehari identities on
$\mathcal{N}^{\mathrm{sc}}_{\varepsilon_n}$ and the Sobolev inequality imply that $\{\mathbf{u}_n\}$ is bounded
in $H$.

We now fix a concentration sequence.
By Lemma~\ref{Lem5.3} (applied, for instance, to the first component), there exist $R>0$, $\eta>0$
and a sequence $\{x_n\}\subset\R^N$ such that
\[
\int_{B_R(x_n)} |u_{1,n}|^{2p}\,dx \ge \eta ~ \text{for all }n .
\]
Moreover, by Lemma~\ref{Lem5.4}, we have
\begin{equation}\label{eq5.29}
	\varepsilon_n x_n \to 0 ~\text{as }n\to\infty .
\end{equation}
To match the notation in Theorem~\ref{thm1.2}, we set
\[
x^{(k)}_{\varepsilon_n,1}=\cdots=x^{(k)}_{\varepsilon_n,m}:=x_n ~ \text{for all }n.
\]
Then \emph{(i)} follows from \eqref{eq5.29}.
\par
\textit{Proof of \emph{(ii)}.}
Define, for $i=1,\dots,m$,
\[
\tilde u_{i,n}(x):=u_{i,n}(x+x_n),~ \tilde{\mathbf u}_n:=(\tilde u_{1,n},\dots,\tilde u_{m,n}).
\]
Since translations preserve the $H$--norm, $\{\tilde{\mathbf u}_n\}$ is bounded in $H$.
Hence, up to a subsequence, there exists $\mathbf U=(U_1,\dots,U_m)\in H$ such that
\[
\tilde{\mathbf u}_n \rightharpoonup \mathbf U ~\text{in }H,
~
\tilde{\mathbf u}_n \to \mathbf U ~\text{in }L^{2p}_{\mathrm{loc}}(\R^N)^m.
\]
In particular, by the choice of $\{x_n\}$ we have $\int_{B_R(0)}|\tilde u_{1,n}|^{2p}\,dx\ge\eta$ for all $n$,
so $U_1\not\equiv 0$ and $\mathbf U\not\equiv \mathbf0$.

\medskip
We now identify the equation solved by $\mathbf{U}$.
Fix $\ell\in\{1,\dots,m\}$ and $\varphi\in C_c^\infty(\mathbb{R}^N)$.
Testing the $\ell$-th Euler--Lagrange equation for $\mathbf{u}_n$ with $\varphi(\cdot-x_n)$ and changing
variables, we obtain
\begin{align*}
	\int_{\mathbb{R}^N}\bigl(\nabla \tilde u_{\ell,n}\cdot\nabla \varphi+\tilde u_{\ell,n}\varphi\bigr)\,dx
	&= \mu_\ell\int_{\mathbb{R}^N} Q\bigl(\varepsilon_n(x+x_n)-y_\ell\bigr)
	|\tilde u_{\ell,n}|^{2p-2}\tilde u_{\ell,n}\varphi\,dx \\
	&~ + \sum_{j\neq \ell}\lambda_{\ell j}\int_{\mathbb{R}^N}
	|\tilde u_{j,n}|^p\,|\tilde u_{\ell,n}|^{p-2}\tilde u_{\ell,n}\varphi\,dx .
\end{align*}
Choose $R_\varphi>0$ such that $\operatorname{supp}(\varphi)\subset B_{R_\varphi}(0)$.
Since $|\varepsilon_n x|\le \varepsilon_n R_\varphi\to0$ for all $x\in B_{R_\varphi}(0)$ and
$\varepsilon_n x_n\to0$ by \eqref{eq5.29}, we have
\[
\varepsilon_n(x+x_n)-y_\ell \to -y_\ell ~\text{uniformly for }x\in B_{R_\varphi}(0),
\]
and hence, by continuity of $Q$,
\[
Q\bigl(\varepsilon_n(x+x_n)-y_\ell\bigr)\to Q(-y_\ell)~\text{uniformly on }B_{R_\varphi}(0).
\]
Together with $\tilde u_{i,n}\to U_i$ in $L^{2p}(B_{R_\varphi}(0))$ for all $i$, we may pass to the limit and conclude
that $\mathbf{U}$ is a weak solution of the limiting system \eqref{eq1.2}.
In particular, $\mathbf U$ is a critical point of $J_0$.

\medskip
We next show that the convergence is in fact strong in $H^1(\R^N)$, namely
\begin{equation}\label{eq5.30}
	u_{i,n}(\cdot+x_n)\to U_i
	~\text{in}~H^1(\mathbb{R}^N)\ \text{as }n\to\infty,~ i=1,\dots,m.
\end{equation}
Indeed, by \eqref{eq5.11} we have $c_{\varepsilon_n,k}\to c_{0,k}$ as $n\to\infty$.
Moreover, applying the concentration--compactness principle to the bounded sequence $\{\tilde{\mathbf u}_n\}$
and using the Br\'ezis--Lieb lemma for the quadratic term, the $2p$--powers and the coupling $p$--powers,
we obtain the energy decomposition along the extracted profile:
\[
\lim_{n\to\infty}J_{\varepsilon_n}(\mathbf u_n)
=\lim_{n\to\infty}J_{0}(\tilde{\mathbf u}_n)
=J_0(\mathbf U)+\lim_{n\to\infty}J_0(\tilde{\mathbf u}_n-\mathbf U),
\]
and the corresponding splitting for the Nehari-type identities.
Since $\mathbf U$ is a critical point of $J_0$, the above identities force
$\lim_{n\to\infty}J_0(\tilde{\mathbf u}_n-\mathbf U)=0$, and hence
$\tilde{\mathbf u}_n\to \mathbf U$ in $H$.
This yields \eqref{eq5.30}.

As a consequence, the truncation maps $u\mapsto u^\pm$ being continuous on $H^1(\R^N)$ give
$u_{i,n}^\pm(\cdot+x_n)\to U_i^\pm$ in $H^1(\R^N)$ for every $i$.
By Lemma~\ref{lem2.3}, we have $u_{i,n}^\pm\not\equiv0$ for all $i$ and all $n$, hence
$U_i^\pm\not\equiv0$ for all $i$. Therefore, $\mathbf{U}$ is fully sign-changing, and \emph{(ii)} follows.

Finally, combining $c_{\varepsilon_n,k}\to c_{0,k}$ with the above strong convergence and the Br\'ezis--Lieb lemma,
we obtain
\begin{equation*}
	c_{0,k}=\lim_{n\to\infty}J_{\varepsilon_n}(\mathbf{u}_n)=J_0(\mathbf{U}).
\end{equation*}
\par	
\textit{Proof of \emph{(iii)}.}
	Fix $i\in\{1,\dots,m\}$ and keep $x_n$ as in Step~1.
	By \eqref{eq5.29} and \eqref{eq5.30}, we have
	\begin{equation*}
		\varepsilon_n x_n\to 0
		~\text{and}~
		u_{i,n}(\cdot+x_n)\to U_i~\text{in}~H^1(\mathbb{R}^N).
	\end{equation*}
	Define
	\[
	v_n(x):=u_{i,n}(x+x_n).
	\]
	Then $v_n\to U_i$ in $H^1(\mathbb{R}^N)$ in particular,
	\begin{equation}\label{eq5.31}
		v_n\to U_i ~\text{strongly in }L^2(\mathbb{R}^N)\ \text{and in }L^{2p}(\mathbb{R}^N).
	\end{equation}
	
	Fix $q\in[1,\infty)$ and $R>0$, and set
	\[
	\Omega_{n,R}:=\mathbb{R}^N\setminus B_{R/\varepsilon_n}(0),
	~
	d_n:=|x_n|.
	\]
	Since $\varepsilon_n d_n=|\varepsilon_n x_n|\to0$, we have $d_n=o(1/\varepsilon_n)$.
	For any $x\in\Omega_{n,R}$ one has $|x|\ge R/\varepsilon_n$, hence
	\[
	|x-x_n|\ge |x|-|x_n|\ge \frac{R}{\varepsilon_n}-d_n,
	\]
	and therefore
	\[
	\Omega_{n,R}-x_n \subset \mathbb{R}^N\setminus B_{R/\varepsilon_n-d_n}(0).
	\]
	Since $R/\varepsilon_n-d_n\to+\infty$, for every $R_0>0$ there exists $n_0$ such that
	$R/\varepsilon_n-d_n\ge R_0$ for all $n\ge n_0$, and consequently
	\begin{equation}\label{eq5.32}
		\Omega_{n,R}-x_n\subset \mathbb{R}^N\setminus B_{R_0}(0)~ \text{for all }n\ge n_0.
	\end{equation}
	
	We claim that $\int_{\Omega_{n,R}}|u_{i,n}|^q\,dx\to0$ for every fixed $R>0$.
	
If $q\in[2,2p]$,
	choose $\theta\in[0,1]$ such that $\frac1q=\frac{\theta}{2}+\frac{1-\theta}{2p}$.
	Then, by interpolation and \eqref{eq5.32}, for all $n\ge n_0$,
	\[
	\|u_{i,n}\|_{L^q(\Omega_{n,R})}
	=\|v_n\|_{L^q(\Omega_{n,R}-x_n)}
	\le \|v_n\|_{L^2(\mathbb{R}^N\setminus B_{R_0})}^{\theta}
	\|v_n\|_{L^{2p}(\mathbb{R}^N\setminus B_{R_0})}^{1-\theta}.
	\]
	Taking $\limsup_{n\to\infty}$ and using \eqref{eq5.31}, we obtain
	\[
	\limsup_{n\to\infty}\|u_{i,n}\|_{L^q(\Omega_{n,R})}
	\le \|U_i\|_{L^2(\mathbb{R}^N\setminus B_{R_0})}^{\theta}
	\|U_i\|_{L^{2p}(\mathbb{R}^N\setminus B_{R_0})}^{1-\theta}.
	\]
	Letting $R_0\to\infty$ yields
	\[
	\limsup_{n\to\infty}\int_{\Omega_{n,R}}|u_{i,n}(x)|^q\,dx=0
	~\text{for every fixed }R>0\text{ and every }q\in[2,2p].
	\]

If $q>2p$, since $\{\mathbf{u}_n\}$ is bounded in $H$ and the system is subcritical, standard elliptic $L^\infty$ estimates yield that there exists $C>0$ independent of $n$ such that
	$\|u_{i,n}\|_{L^\infty(\mathbb{R}^N)}\le C$ for all $n$.
	Hence,
	\[
	\int_{\Omega_{n,R}}|u_{i,n}|^q\,dx
	\le \|u_{i,n}\|_{L^\infty(\mathbb{R}^N)}^{\,q-2p}\int_{\Omega_{n,R}}|u_{i,n}|^{2p}\,dx
	\le C^{q-2p}\int_{\Omega_{n,R}}|u_{i,n}|^{2p}\,dx\to0.
	\]
If $q\in[1,2]$, fix $\sigma\in(0,1)$ and set $v_n(x):=u_{i,n}(x+x_n)$.
Since $\{u_n\}$ is bounded in $H$ and the right-hand side has subcritical growth,
standard Agmon-type weighted estimates for $-\Delta w+w=g$ applied componentwise yield
a constant $C_\sigma>0$, independent of $n$, such that
\begin{equation}\label{eq5.33}
	\int_{\R^N}e^{2\sigma|x|}\bigl(|\nabla v_n|^2+|v_n|^2\bigr)\,dx\le C_\sigma .
\end{equation}
Let $R_0>0$ and $E_{R_0}:=\R^N\setminus B_{R_0}(0)$. For $q\in[1,2)$, writing
$|v_n|^q=(e^{\sigma|x|}|v_n|)^q e^{-\sigma q|x|}$ and applying H\"older's inequality, we obtain
\[
\int_{E_{R_0}}|v_n|^q\,dx
\le
\Bigl(\int_{\R^N}e^{2\sigma|x|}|v_n|^2\,dx\Bigr)^{q/2}
\Bigl(\int_{E_{R_0}}e^{-\frac{2\sigma q}{2-q}|x|}\,dx\Bigr)^{(2-q)/2}.
\]
By \eqref{eq5.33} the first factor is uniformly bounded in $n$, while the second factor
tends to $0$ as $R_0\to\infty$ since the exponential weight is integrable on $\R^N$. Hence,
\[
\lim_{R_0\to\infty}\ \limsup_{n\to\infty}\int_{\R^N\setminus B_{R_0}(0)}|v_n(x)|^q\,dx=0,
\qquad q\in[1,2).
\]
	For $q=2$, we simply write
	\[
	\int_{\mathbb{R}^N\setminus B_{R_0}(0)}|v_n(x)|^2\,dx
	\le e^{-2\sigma R_0}\int_{\mathbb{R}^N}e^{2\sigma|x|}|v_n(x)|^2\,dx
	\le e^{-2\sigma R_0} C_\sigma,
	\]
	and hence the same conclusion holds for $q=2$.
	Consequently,
	\[
	\lim_{R_0\to+\infty}\ \limsup_{n\to\infty}\int_{\mathbb{R}^N\setminus B_{R_0}(0)}|v_n(x)|^q\,dx=0
	~\text{for every }q\in[1,2].
	\]
	
	Finally, by \eqref{eq5.32}, $\Omega_{n,R}-x_n\subset\mathbb{R}^N\setminus B_{R/\varepsilon_n-d_n}(0)$
	with $R/\varepsilon_n-d_n\to+\infty$, we deduce
	\[
	\limsup_{n\to\infty}\int_{\Omega_{n,R}}|u_{i,n}(x)|^q\,dx
	=
	\limsup_{n\to\infty}\int_{\Omega_{n,R}-x_n}|v_n(x)|^q\,dx
	\le
	\limsup_{n\to\infty}\int_{\mathbb{R}^N\setminus B_{R/\varepsilon_n-d_n}(0)}|v_n(x)|^q\,dx\to0,
	\]
	for every fixed $R>0$ and every $q\in[1,2]$.  Together with the already proved cases
	$q\in[2,2p]$ and $q>2p$, this yields (iii) for all $q\in[1,\infty)$.
	
	Altogether, for every $q\in[1,\infty)$ and every fixed $R>0$,
	\[
	\limsup_{n\to\infty}\int_{\mathbb{R}^N\setminus B_{R/\varepsilon_n}(0)}|u_{i,n}(x)|^q\,dx=0,
	\]
	and therefore
	\[
	\lim_{R\to\infty}\ \limsup_{n\to\infty}
	\int_{\mathbb{R}^N\setminus B_{R/\varepsilon_n}(0)}|u_{i,n}(x)|^q\,dx=0,
	~\forall\,q\in[1,\infty).
	\]
	
	It remains to prove the case $q=\infty$.
	Assume by contradiction that there exist $\delta_0>0$, $R_n\to\infty$ and $z_n\in\mathbb{R}^N\setminus B_{R_n/\varepsilon_n}(0)$ such that
	\[
	|u_{i,n}(z_n)|\ge \delta_0 .
	\]
	Define
	\[
	\hat{\mathbf{u}}_n(x):=\mathbf{u}_n(x+z_n)=(\hat u_{1,n}(x),\dots,\hat u_{m,n}(x)).
	\]
	Then $\{\hat{\mathbf{u}}_n\}$ is bounded in $H$ and $|\hat u_{i,n}(0)|\ge\delta_0$ for all $n$.

	Fix $R>0$ and consider the system satisfied by $\hat{\mathbf u}_n$ on $B_R(0)$
	\[
	-\Delta \hat u_{\ell,n}+\hat u_{\ell,n}
	=\mu_\ell\,Q(\varepsilon_n(x+z_n)-y_\ell)\,|\hat u_{\ell,n}|^{2p-2}\hat u_{\ell,n}
	+\sum_{j\neq \ell}\lambda_{\ell j}|\hat u_{j,n}|^{p}|\hat u_{\ell,n}|^{p-2}\hat u_{\ell,n},~\ell=1,\dots,m.
	\]
	Since $Q\in L^\infty(\R^N)$ and $\{\hat{\mathbf u}_n\}$ is bounded in $H$, we have
	$\{\hat u_{\ell,n}\}$ bounded in $L^{2^*}(B_R)$ for every $\ell$.
	Hence the right-hand side is bounded in $L^{t}(B_R)$ for some $t>1$ independent of $n$, and therefore,
	by standard local elliptic $W^{2,t}$-estimates (\cite{GilbargTrudinger1983}) we obtain that
	$\{\hat u_{\ell,n}\}$ is bounded in $W^{2,t}(B_{R/2})$ for each $\ell$.
	Choosing $t>N$, the Sobolev--Morrey embedding (\cite{GilbargTrudinger1983}) yields that, up to a subsequence,
	\[
	\hat u_{\ell,n}\to \hat W_\ell ~\text{in }C^{1,\alpha}(B_{R/2})\ \text{for some }\alpha\in(0,1),
	~ \ell=1,\dots,m,
	\]
	and in particular
	\[
	\hat u_{i,n}(0)\to \hat W_i(0)~\text{as }n\to\infty.
	\]
	Since $|\hat u_{i,n}(0)|\ge\delta_0$ for all $n$, it follows that $\hat W_i(0)\neq0$.
	
	Moreover, by boundedness in $H$, up to a subsequence, we may assume that
	\[
	\hat{\mathbf{u}}_n\rightharpoonup \hat{\mathbf{W}}\ \text{in }H,
	~
	\hat{\mathbf{u}}_n\to \hat{\mathbf{W}}\ \text{in }L^r_{\mathrm{loc}}(\mathbb{R}^N)^m
	\ \text{for every }r\in[2,2^*),
	\]
	where $\hat{\mathbf W}=(\hat W_1,\dots,\hat W_m)$.
	
	\medskip
	Since $|z_n|\ge R_n/\varepsilon_n$, we have $|\varepsilon_n z_n|\ge R_n\to\infty$.
	Let $K\subset\mathbb{R}^N$ be bounded. Then $\varepsilon_n x\to0$ uniformly on $K$, and thus for each $\ell$,
	\[
	|\varepsilon_n(x+z_n)-y_\ell|
	\ge |\varepsilon_n z_n|-|y_\ell|-\varepsilon_n|x|
	\to\infty ~\text{uniformly for }x\in K.
	\]
	By $(A_3)$, there exists $R_Q>0$ such that $\operatorname{supp}(Q)\subset B_{R_Q}(0)$.
	Hence, for all large $n$ and all $x\in K$,
	$|\varepsilon_n(x+z_n)-y_\ell|>R_Q$, which implies
	\[
	Q(\varepsilon_n(x+z_n)-y_\ell)=0~\text{on }K \ \text{for all large }n \text{ and all }\ell.
	\]
	Passing to the limit in the weak formulation, using the local strong convergence above, we infer that
	$\hat{\mathbf{W}}$ solves
	\[
	-\Delta \hat W_\ell+\hat W_\ell=\sum_{j\neq \ell}\lambda_{\ell j}|\hat W_j|^p|\hat W_\ell|^{p-2}\hat W_\ell,
	~ \ell=1,\dots,m.
	\]
	
	Testing each equation with $\hat W_\ell$ and summing over $\ell$ yields
	\[
	\sum_{\ell=1}^m\int_{\mathbb{R}^N}\bigl(|\nabla \hat W_\ell|^2+|\hat W_\ell|^2\bigr)\,dx
	=\sum_{\ell=1}^m\sum_{j\neq\ell}\lambda_{\ell j}\int_{\mathbb{R}^N}|\hat W_\ell|^p|\hat W_j|^p\,dx
	\le 0,
	\]
	because $\lambda_{\ell j}<0$ by $(A_2)$. Hence $\hat{\mathbf{W}}=\mathbf{0}$, which contradicts
	$\hat W_i(0)\neq0$. Therefore,
	\[
	\lim_{R\to\infty}\ \limsup_{n\to\infty}
	\|u_{i,n}\|_{L^\infty(\mathbb{R}^N\setminus B_{R/\varepsilon_n}(0))}=0,
	\]
which implies \emph{(iii)} holds.
	
	\par
	\textit{Proof of \emph{(iv)}.}
	In Steps~1--2 we fixed $k$ and omitted the superscript $(k)$.
	In this step $k$ varies, so we restore the superscript to avoid ambiguity.
	Since $\{c_{0,k}\}_{k\ge1}$ is nondecreasing and unbounded, we can choose a strictly increasing subsequence
	$\{c_{0,k_j}\}_{j\ge1}$ such that
	\[
	0<c_{0,k_1}<c_{0,k_2}<\cdots<c_{0,k_j}<\cdots\to+\infty.
	\]
	For simplicity of notation, we relabel this subsequence and still denote it by $\{c_{0,k}\}_{k\ge1}$.
	
	Fix $k\ge1$. Choose an arbitrary sequence $\{\varepsilon_n^{(k)}\}\subset(0,+\infty)$ with
	$\varepsilon_n^{(k)}\downarrow0$, and let $\mathbf{u}_{\varepsilon_n^{(k)}}^{(k)}$ be the $k$-th sign-changing
	solution given by Theorem~\ref{thm1.1}, namely
	\[
	J_{\varepsilon_n^{(k)}}\bigl(\mathbf{u}_{\varepsilon_n^{(k)}}^{(k)}\bigr)=c_{\varepsilon_n^{(k)},k}.
	\]
	Applying Steps~1--2 to the sequence $\{\mathbf{u}_{\varepsilon_n^{(k)}}^{(k)}\}_n$, we obtain a limiting profile
	$\mathbf{U}^{(k)}\in H$ such that $\mathbf{U}^{(k)}$ is a fully sign-changing weak solution of
	\eqref{eq1.2} and
	\[
	J_0(\mathbf{U}^{(k)})
	=\lim_{n\to\infty}J_{\varepsilon_n^{(k)}}\bigl(\mathbf{u}_{\varepsilon_n^{(k)}}^{(k)}\bigr)
	=\lim_{n\to\infty}c_{\varepsilon_n^{(k)},k}
	=c_{0,k},
	\]
	where the last equality follows from \eqref{eq5.11}.
	In particular, if $k\neq \ell$, then
	\[
	J_0(\mathbf{U}^{(k)})=c_{0,k}\neq c_{0,\ell}=J_0(\mathbf{U}^{(\ell)}),
	\]
	and hence $\mathbf{U}^{(k)}\neq \mathbf{U}^{(\ell)}$.
	Consequently, the family $\{\mathbf{U}^{(k)}\}_{k\ge1}$ is pairwise distinct and satisfies
	\[
	J_0(\mathbf{U}^{(1)})<J_0(\mathbf{U}^{(2)})<\cdots\to+\infty.
	\]
	This proves \emph{(iv)}.

\par
{\bf Conflict Of Interest Statement.} The authors declare that there are no conflict of interests, we do not have any possible conflicts of interest.
\par
{\bf Data Availability Statement.} Our manuscript has non associated data.

\end{document}